\documentclass[12pt]{amsart}
\usepackage{amsmath, amsfonts, amsbsy, amsthm, amscd, graphicx}
\usepackage{amssymb, latexsym, color}

\numberwithin{equation}{section}

\theoremstyle{plain}
\newtheorem{Theorem}{Theorem}[section]
\newtheorem{Proposition}[Theorem]{Proposition}
\newtheorem{Lemma}[Theorem]{Lemma}
\newtheorem{Corollary}[Theorem]{Corollary}

\theoremstyle{definition}
\newtheorem{Definition}[Theorem]{Definition}

\theoremstyle{remark}
\newtheorem{Remark}{{\bf Remark}}
\newtheorem{Example}{Example}

\newcommand{\C}{{\mathbb C}}
\newcommand{\R}{{\mathbb R}}
\newcommand{\Q}{{\mathbb Q}}
\newcommand{\F}{{\mathbb F}}
\newcommand{\Z}{{\mathbb Z}}
\newcommand{\N}{{\mathbb N}}

\newcommand{\isom}{\mathop{\rm Isom}\nolimits}

\newcommand{\bary}{\mathop{\rm bar}\nolimits}

\newcommand{\diam}{\mathop{\rm diam}\nolimits}

\newcommand{\cat}{\mathrm{CAT}(0)}

\newcommand{\girth}{\mathrm{girth}} 
\newcommand{\Dirac}{\mathrm{Dirac}}

\begin{document}

\title{$N$-step energy of 
maps and fixed-point property of 
random groups} 

\author{Hiroyasu Izeki$\mbox{}^1$} 
\thanks{$\mbox{}^1$Department of Mathematics, Faculty of Science and Technology, Keio University, 
Kohoku-ku, Yokohama 223-8522, Japan, izeki@math.keio.ac.jp}

\author{Takefumi Kondo$\mbox{}^2$}
\thanks{$\mbox{}^2$Department of Mathematics, Graduate School of Science, Kobe University, 
Kobe 657-8501, Japan, takefumi@math.kobe-u.ac.jp} 

\author{Shin Nayatani$\mbox{}^3$*}
\thanks{$\mbox{}^3$Graduate School of Mathematics, Nagoya University,
Chikusa-ku, Nagoya 464-8602, Japan, nayatani@math.nagoya-u.ac.jp}
\thanks{* Corresponding author}


\maketitle

\begin{abstract}
We prove that a random group of the graph model associated with a sequence of expanders 
has fixed-point property for a certain class of $\cat$ spaces. 
We use Gromov's criterion for fixed-point property in terms of the growth 
of $n$-step energy of equivariant maps from a finitely generated group into a $\cat$ space, 
to which we give a detailed proof.  
We estimate a relevant geometric invariant of the tangent cones of the Euclidean buildings 
associated with the groups $\mathrm{PGL}(m,\Q_r)$, and deduce from the general result above that
the same random group has fixed-point property for all of these Euclidean buildings 
with $m$ bounded from above. 
\end{abstract} 

\section*{Introduction} 

Random groups were introduced by Gromov \cite{Gromov2} as a framework in which he justified his previous claim 
that `most' discrete groups are hyperbolic \cite{Gromov1}. 
While this standard model, called the density model, of random groups has been actively studied, 
Gromov \cite{Gromov3} introduced another model, called the graph model, of random groups, in search for 
infinite groups which cannot be uniformly embedded into Hilbert spaces, thereby being a counterexample 
to a version of Baum-Connes conjecture. 
Note that the graph model is formed by choosing an infinite sequence of finite graphs with increasing vertices, 
and Gromov chose
a sequence of (bounded-degree) expanders, satisfying some additional conditions 
so that the corresponding random group was
non-elementary  hyperbolic, hence infinite. 
Throughout the introduction, we assume that this choice is made and fixed. 
In the same paper, Gromov claimed that a random group of the graph model had fixed-point property for 
all Hadamard manifolds (possibly of infinite dimensions). 
Here we say that a group $\Gamma$ has fixed-point property for a metric space $Y$ if for any homomorphism 
$\rho\colon \Gamma \rightarrow \isom(Y)$, $\rho(\Gamma)$ has a global fixed point in $Y$. 
If $\mathcal{Y}$ is a family of metric spaces and $\Gamma$ has fixed-point property for all members of $\mathcal{Y}$, 
we say that $\Gamma$ has fixed-point property for $\mathcal{Y}$. 
Silberman \cite{Silberman} then rigorously proved that the same random group 
had fixed-point property 
for Hilbert spaces, which was equivalent to saying that the random group had Kazhdan's property (T). 
We refer the reader to Ollivier's monograph \cite{Ollivier} for extensive information on random groups. 

In the present paper, we prove that a random group of the graph model has fixed-point property 
for a certain class of $\cat$ spaces, including all Hadamard manifolds. 
We therefore justify the above mentioned claim of Gromov in a generalized form. 
To state our main result more in detail, recall that we \cite{Izeki-Nayatani} introduced a certain geometric invariant, 
denoted by $\delta$, of $\cat$ space which takes values in the interval $[0,1]$. 
It is worth mentioning that the invariant of a $\cat$ space can be computed as the supremum of 
its values for all tangent cones of the space. 
For $0\leq \delta_0< 1$, let $\mathcal{Y}_{\leq\delta_0}$ denote the class of $\cat$ spaces $Y$ satisfying 
$\delta(Y)\leq \delta_0$, or equivalently $\delta(TC_pY)\leq \delta_0$ for all $p\in Y$. 
Then a random group is infinite hyperbolic and has fixed-point property for all members of $\mathcal{Y}_{\leq\delta_0}$. 
This result is compared to the authors' previous result that if $\delta_0<1/2$, a random group in Zuk's triangular model 
has fixed-point property for all members of $\mathcal{Y}_{\leq\delta_0}$. 

As in Silberman's, our proof is built of two parts, one geometric and the other probabilistic. 
The probalisitic part follows Silberman's argument mostly verbatim, but our presentation has some advantages. 
First, we simplify his argument 
by replacing the large deviation inequality for the Bernoulli walk he used by the central limit theorem. 
This also enables us to allow degree two vertices in the graphs and therefore state our result in a form 
applicable to subdivided expanders. 
Secondly, we generalize Silberman's spectal gap inequality for maps from a finite graph into a Hilbert space 
to the inequality for maps with $\cat$ targets, and this 
enables us to state the result for more general $\cat$ spaces than Hilbert spaces. 
The geometric part of the proof is completely different from Silberman's; we use Gromov's criterion 
for fixed-point property in terms of the growth of $n$-step energy of equivariant maps from a group 
into a $\cat$ space. 
Since Gromov does not give a detailed proof to this result, we undertake to do so. 

With the general result at hand, it is important to compute or estimate from above the invariant $\delta$ 
of a $\cat$ metric cone. 
To do this, we relate it to a modified version of distortion of the cone, which we call the radial distortion. 

The following fact is well-known. 
Suppose that a discrete group $\Gamma$ has fixed-point property for Hilbert spaces, 
all symmetric spaces associated with the groups $\mathrm{PGL}(m, \R)$ and $\mathrm{PGL}(m, \C)$, and all 
Euclidean buildings associated with the groups $\mathrm{PGL}(m,\Q_r)$ with $r$ prime. 
Then $\Gamma$ is nonlinear, in the sense that it admits no faituful linear representation, and more strongly, 
any finite-dimensional linear representation of $\Gamma$ has finite image. 
It is therefore interesting to see which of these spaces are in the above class of $\cat$ spaces. 
Since Hilbert spaces and the symmetric spaces are Hadamard manifolds and hence trivially belong to this class,  
it remains to investigate the Euclidean buildings. 
As mentioned above, the estimation of the invariant $\delta$ of their tangent cones can be reduced to that of 
the radial distortion of these cones. 
By carrying out the latter task, we conclude that the invariant $\delta$ of the tangent cones of the Euclidean buildings 
are bounded from above by a constant less than one which depends only on the dimensions of the buildings. 
Combining this and the general result above, we finally conclude that a random group of the graph model has 
fixed-point property for all of the Euclidean buildings with dimensions bounded from above by a positive integer, 
specified in advance. 

This paper is organized as follows. 
In Section 1 we prepare some definitions, notations and results concerning $\cat$ spaces which are necessary 
in later sections. 
In Section 2 under the situation that an isometric action of a finitely generated group on a $\cat$ space 
is given, we study the $n$-step energy of an equivariant map from the group into the $\cat$ space. 
After some preliminaries, we give a proof to Gromov's result which states that if the $n$-step energy grows 
strictly slower than $n$ times the (single-step) energy, then the action is forced to have a global fixed point. 
In fact, we prove the result in a slightly generalized form, which will be necessary in the next section. 
We then treat the special case that the target space is an affine Hilbert space, and conclude this section by proving 
a sort of converse of Gromov's result holds for some $\cat$ spaces. 
In Section 3 we first recall the formalism of the graph model of random groups, and then prove a fixed-point theorem 
for random groups, a version of which will be stated in terms of the invariant $\delta$. 
In Section 4 we first give an upper bound for the invariant $\delta$ of a $\cat$ metric cone, restricted to 
measures with barycenter at the cone point, in terms of the radial distortion of the cone. 
We then estimate $\delta$ of the tangent cones of the Euclidean buildings associated with the groups 
$\mathrm{PGL}(m,\Q_r)$, by estimating the radial distortion of the cones. 

Part of this paper was announced in \cite{Izeki, Kondo1}. 

\section{Preliminaries on $\cat$ spaces}

In this section, we briefly recall some definitions and results
concerning $\cat$ spaces.  We refer the reader to \cite{Bridson-Haefliger} for a
detailed exposition on the subject.  We follow the notations
used in \cite[\S 1]{Izeki-Nayatani}. 

Let $Y$ be a metric space and $p,q \in Y$.
A {\it geodesic} joining $p$ to $q$ is 
a map $c \colon [0,l] \longrightarrow Y$ satisfying 
$c(0)=p$, $c(l)=q$ and 
$d(c(t),c(t'))=|t-t'|d(p,q)/l$ for any $t,t' \in [0,l]$. 
A geodesic $c \colon [0,l] \longrightarrow Y$ with $l=d(c(0),c(l))$ is
called {\it unit speed}; a unit speed geodesic is nothing but an
isometric embedding of an interval. 
We say that $Y$ is a {\it geodesic space} if any two points in $Y$ 
are joined by a geodesic.  

Consider a triangle in $Y$ whose vertices are $p_1,p_2,p_3 \in Y$ and 
sides are three geodesic segments $p_1p_2, p_2p_3, p_3p_1$ joining
pairs of these vertices. 
We denote this triangle by $\Delta(p_1,p_2,p_3)$ and
call such a triangle a {\it geodesic triangle}.  
Take
a triangle $\Delta(\overline{p_1},\overline{p_2},\overline{p_3})$ in
$\R^2$ with the same side lengths: 
$d_{\R^2}(\overline{p_i},\overline{p_j})=d_Y(p_i, p_j)$.
We call $\Delta(\overline{p_1},\overline{p_2},\overline{p_3})$ a 
{\it comparison triangle} for $\Delta(p_1,p_2,p_3)$. 
A point $\overline{q}\in \overline{p_i}\overline{p_j}$
is called a {\it comparison point} for $q\in p_ip_j$ if 
$d_Y(p_i,q) = d_{\R^2}(\overline{p_i}, \overline{q})$.
A geodesic triangle $\Delta(p_1,p_2,p_3)$ in $Y$ is said to satisfy the
{\it $\cat$ condition} if 
$d_Y(q_1,q_2) \leq d_{\R^2}(\overline{q_1}, \overline{q_2})$ for any
pair of points $q_1, q_2$ on the sides of $\Delta(p_1,p_2,p_3)$ and
their comparison points $\overline{q_1}, \overline{q_2}$. 
A geodesic space $Y$ is called a {\it $\cat$ space} if 
every geodesic triangle in $Y$ satisfies the $\cat$ condition. 
Roughly speaking, a $\cat$ space is a geodesic space all of whose
geodesic triangles are thinner than Euclidean triangles. 


Note that, for a $\cat$ space $Y$, the uniqueness of a unit
speed geodesic joining any pair of points in $Y$, and the
contractibility of $Y$ immediately follow from the definition. 
Throughout this paper, we assume metric spaces under consideration are 
complete. (We should point out here that a complete $\cat$ space was 
given a distinguished name `Hadamard space' in
\cite{Izeki-Nayatani}, though we will not use this terminology in the
present paper.)  

The following is a characterization of $\cat$ spaces, which is often used
as an alternative definition of $\cat$ spaces. 

\begin{Proposition}{\cite
{Bridson-Haefliger}}
\label{CAT(0)_ineq}
 A geodesic space $Y$ is a $\cat$ space if and only if,  
 for any $p \in Y$ and any 
 geodesic $c \colon [0,1]\longrightarrow Y$, 
 \begin{equation*}
  d(p, c(t))^2 \leq 
  (1-t)d(p,c(0))^2 + td(p,c(1))^2 - t(1-t)d(c(0),c(1))^2
 \end{equation*}
 holds. 
\end{Proposition}

\begin{Proposition}
\label{barycenter_prop}
Let $Y$ be a $\cat$ space, and $\nu$ a probability measure on $Y$. 
Suppose that the integral 
$$
\int_Y d(p,q)^2 d\nu(p) 
$$
is finite for some {\rm (}hence any{\rm )} point $q\in Y$. 
Then there exists a unique point $p_0 \in Y$ which minimizes the function
$$
 q \mapsto \int_Y d(p,q)^2 d\nu(p),\quad q\in Y.
$$
\end{Proposition}

For a proof, see \cite[p. 639, Lemma 2.5.1] {Korevaar-Schoen}. 
We call the point $p_0$ the {\it barycenter of} 
$\nu$ and denote it by $\overline{\nu}$ or ${\rm bar}(\nu)$. 
Mostly, we will consider a measure $\nu$ with finite support; 
$\nu$ is given as a convex combination 
$\nu = \sum_{i=1}^m t_i \Dirac_{p_i}$ of Dirac measures 
$\Dirac_{p_i}$'s, where $\sum_{i=1}^m t_i=1$ and $t_i\geq 0$ for
$i=1,\dots, m$.  
In such a case, we often say $\overline{\nu}$ is the {\it barycenter of} 
$\{p_1, \dots, p_m\}$ {\it with weight} $\{t_1, \dots, t_m\}$.  

\begin{Definition}
\label{tangent cone}
Let $Y$ be a $\cat$ space. \\ 
{\rm (1)}\quad Let $c$ and $c'$ be two nontrivial 
geodesics in $Y$ starting 
from $p \in Y$. The {\it angle} $\angle_p(c,c')$ between $c$ and $c'$ 
is defined by 
\begin{equation*}
 \angle_p(c,c')= \lim_{t,t' \rightarrow 0}
 \angle_{\overline{p}}(\overline{c(t)},\overline{c'(t')}),
\end{equation*}
 where
 $\angle_{\overline{p}}(\overline{c(t)},\overline{c'(t')})$ 
 denotes the angle between the sides
 $\overline{p}\overline{c(t)}$ and $\overline{p}\overline{c'(t)}$
 of the comparison triangle
 $\Delta(\overline{p},\overline{c(t)},\overline{c'(t')})
 \subset \R^2$.\\ 
{\rm (2)}\quad Let $p \in Y$.
 We define an equivalence relation  $\sim$ on the set of
 nontrivial geodesics starting from $p$ by 
 $c \sim c' \Longleftrightarrow \angle_p(c,c')=0$. 
 Then the angle $\angle_p$ induces a distance on the quotient 
 $(S_pY)^{\circ} = \{\text{nontrivial geodesics starting from } 
 p \}/\sim$, 
 which we denote by the same symbol $\angle_p$.  The completion 
 $(S_pY,\angle_p)$ of the metric space $((S_pY)^{\circ}, \angle_p)$ is
 called the {\it space of directions} 
 at $p$.\\ 
{\rm (3)}\quad Let $TC_pY$ be the cone over $S_pY$, namely,
\begin{equation*}
 TC_pY = (S_pY \times \R_+) / (S_pY \times \{0\}). 
\end{equation*}
 Let $v,v' \in TC_pY$. We may write $v=(u,t)$ and $v'=(u',t')$, where
 $u,u' \in S_pY$ and $t,t' \in \R_{+}$.   Then
\begin{equation*}
 d_{TC_pY}(v, v')= t^2 + {t'}^2 - 2tt'\cos \angle_p(u,u')
\end{equation*}
 defines a distance on $TC_pY$. The metric space $(TC_pY, d_{TC_pY})$ is
 again a $\cat$ space and is called the {\it tangent cone} of $Y$ at $p$.  We
 define an `inner product'  on  $TC_pY$ by
\begin{equation*}
 \langle v, v' \rangle = tt'\cos \angle_p(u,u').
\end{equation*}
 We often denote the length $t$ of $v$ by $|v|$; thus we have
 $|v|=\sqrt{\langle v,v \rangle}=d_{TC_pY}(0_p,v)$, where $0_p$ denotes
 the cone point, which is the equivalence class of 
 $(u,0) \in S_pY \times \R_{+}$ in $TC_pY$.   

{\rm (4)}\quad Define a map $\pi_p \colon Y \longrightarrow TC_pY$ by
 $\pi_p(q)=([c], d_Y(p,q))$, where $c$ is the geodesic
 joining $p$ to $q$ and $[c]\in S_pY$ is the equivalence class of
 $c$. Then $\pi_p$ is distance-nonincreasing.  
\end{Definition}

A complete, simply connected Riemannian manifold $Y$ with nonpositive
sectional curvature, often called a {\it Hadamard manifold}, is a
typical example of $\cat$ space. 
For such a $Y$, $S_pY$ (resp. $TC_pY$) is the unit tangent sphere
(resp. the tangent space) at $p$. 
The map $\pi_p$ is the inverse of the exponential map.
Hilbert spaces, metric trees and Euclidean buildings supply other examples 
of $\cat$ spaces (see \S \ref{section_delta} for Euclidean buildings).

\section{The $n$-step energy of equivariant maps}

Let $\Gamma$ be a finitely generated group, and $Y$ a $\cat$ space. 
Suppose a homomorphism $\rho\colon \Gamma \longrightarrow \isom(Y)$ is given. 
In \cite{Gromov3}, Gromov formulated a sufficient condition for $\rho(\Gamma)$ to have a
global fixed point (i.e., there exists a point $p \in Y$ such that $\rho(\Gamma)p =p$) in terms of 
the growth of $n$-step energy of $\rho$-equivariant maps. 
The purpose of this section is to give a detailed proof of Gromov's result. 

We consider a random walk on $\Gamma$ given by transition probability measures 
$\{ \mu(\gamma, \cdot) \}_{\gamma\in\Gamma}$ on $\Gamma$ which is  $\Gamma$-invariant,
finitely supported, symmetric, and irreducible. 
In other words, we are given a nonnegative function $\mu$ on $\Gamma \times \Gamma$ satisfying 
\begin{enumerate}
\renewcommand{\labelenumi}{(\alph{enumi})} 
 \item  $\mu(\gamma \gamma', \gamma \gamma'')=\mu(\gamma',\gamma'')$
 for any $\gamma$, $\gamma'$, and $\gamma'' \in \Gamma$,
 \item  for any $\gamma \in \Gamma$, $\mu(\gamma, \gamma')= 0$ for
 all but finitely many $\gamma' \in \Gamma$, 
 \item  for any $\gamma \in \Gamma$, 
        $\displaystyle{\sum_{\gamma' \in \Gamma} \mu(\gamma, \gamma') = 1}$, 
 \item  $\mu(\gamma, \gamma')=\mu(\gamma',\gamma)$ for any 
        $\gamma, \gamma' \in \Gamma$, 

 \item  for any $\gamma$, $\gamma'\in \Gamma$, there exist 
        $\gamma_0,\gamma_1, \dots, \gamma_n \in \Gamma$ such that
        $\gamma=\gamma_0$, $\gamma'=\gamma_n$, and
        $\mu(\gamma_i,\gamma_{i+1})\not=0$, $i=0,\dots, n-1$. 
\end{enumerate} 
The last condition is called the {\it irreducibility} of a random walk
and means that $\Gamma$ is  `connected' with respect
to $\mu$, that is,  for any pair of points in $\Gamma$, one can move
from one to the other with positive probability. 
Though we could begin with a discrete countable group $\Gamma$, the existence of such a $\mu$ 
would force $\Gamma$ to be finitely generated. 

We say a map $f\colon \Gamma \longrightarrow Y$ is {\it $\rho$-equivariant} 
if $f$ satisfies $f(\gamma \gamma')= \rho(\gamma)f(\gamma')$ for all $\gamma,\ \gamma' \in \Gamma$. 
(Here we regard $\Gamma$ itself as a space with left $\Gamma$-action.) 
We define the {\it energy} $E_{\mu,\rho}(f)$ of a $\rho$-equivariant map $f$ by 
\begin{equation}\label{energy_def}
E_{\mu,\rho}(f)= \frac{1}{2} \sum_{\gamma' \in \Gamma} 
\mu(\gamma,\gamma')\, d_Y(f(\gamma),f(\gamma'))^2, 
\end{equation} 
where $\gamma$ is an arbitrarily chosen element of $\Gamma$. 
Note that since $f$ is $\rho$-equivariant and $\mu$ is $\Gamma$-invariant, 
the right-hand side of \eqref{energy_def} does not depend on the particular choice of $\gamma$. 
It  is often convenient to choose $\gamma= e$, the identity element of $\Gamma$. 
A $\rho$-equivariant map $f$ is said to be {\it harmonic} if $f$
minimizes $E_{\mu,\rho}$ among all $\rho$-equivariant maps. 
Note that the image of a $\rho$-equivariant map 
$f\colon \Gamma \longrightarrow Y$ is the $\rho(\Gamma)$-orbit of the point
$f(e)$, and $f$ is determined by the choice of $f(e)\in Y$. Therefore, the
set of all $\rho$-equivariant maps from $\Gamma$ to $Y$, denoted by $\mathcal{M}_\rho$, 
can be identified with $Y$.  Then the energy functional $E_{\mu,\rho}$ becomes a
convex continuous function on $\mathcal{M}_\rho\cong Y$. 

We define the {\it link} $L_{\gamma}$ of $\gamma \in \Gamma$ with respect to $\mu$ by
$L_{\gamma}=\{ \gamma' \in \Gamma \mid \mu(\gamma,\gamma')>0\}$, and for a $\rho$-equivariant map $f$, 
define a map $F_{\gamma}\colon L_{\gamma} \longrightarrow TC_{f(\gamma)}Y$ by 
$F_{\gamma}(\gamma')=\pi_{f(\gamma)} (f(\gamma'))$, 
where $TC_pY$ is the tangent cone of $Y$ at $p$ and $\pi_{p}\colon Y \longrightarrow TC_{p}Y$ 
is the natural projection. 
Denote by $-\Delta_\mu f(e) \in TC_{f(e)}Y$ the barycenter of the push-forward measure $(F_e)_* (\mu(e,\cdot))$. 
Then a $\rho$-equivariant map $f$ is harmonic if and only if $-\Delta_\mu f(e)=0_{f(e)}$. 
Note that $2(- \Delta_\mu f(e))$ should be interpreted as the negative of the gradient of 
$E_{\mu,\rho}$ at $f$.
(Indeed, they coincide when $Y$ is a Riemannian manifold. 
See \cite{Izeki-Kondo-Nayatani}, \cite{Izeki-Nayatani}.) 

The following proposition gives a sufficient 
condition for the existence of a fixed-point of $\rho(\Gamma)$ in terms of
the energy functional. 

\begin{Proposition}[\cite{Izeki-Kondo-Nayatani}, \cite{Izeki-Nayatani}]
\label{gradient}
 Let $\Gamma$ be a finitely generated group equipped with a
 $\Gamma$-invariant, finitely supported, symmetric, and irreducible 
 random walk $\mu$. Let $Y$ 
 be a $\cat$ space and $\rho\colon \Gamma \longrightarrow \isom(Y)$ a
 homomorphism. Suppose there is a positive constant $C$ such that
 $|-\Delta_\mu f(e) |^2 \geq C E_{\mu,\rho}(f)$ holds for every
 $\rho$-equivariant map $f$. Then $\rho(\Gamma)$ admits a global fixed
 point. 
\end{Proposition}

In fact, under the assumption, $|-\Delta_\mu f_t (e)|$
decreases to $0$  
rapidly along the Jost-Mayer's gradient flow $f_t$
of $E_{\mu,\rho}$, and is integrable on $[0,\infty)$. 
(See \cite{Jost} and \cite{Mayer} for the Jost-Mayer's gradient flow.)
This means that the length of the flow starting from $f_0=f$ is
finite up to time infinity.  In particular, by taking a divergent sequence  
$\{t_i\}_{i\in \N}\subset \R$, we obtain a Cauchy sequence 
$\{f_{t_i}\}_{i \in \N} \subset \mathcal{M}_{\rho}$ and 
a $\rho$-equivariant map $f_{\infty}$ as its limit. 
Since $|-\Delta_\mu f_t (e)|\to 0$, under the assumption again, 
we see that $f_{\infty}$ satisfies $E_{\mu,\rho}(f_{\infty})=0$,
which implies that $d(f(\gamma), f(\gamma'))=0$ whenever 
$\mu(\gamma,\gamma')\not= 0$. Since $\mu$ is irreducible, any pair of
elements in $\Gamma$ can be connected by a path consisting of segments
of the form $(\gamma, \gamma')$ such that $\mu(\gamma, \gamma')\not=0$. 
Therefore, 
$f_{\infty}$ must be a constant 
map. (Actually, the irreducibility of 
$\mu$ is necessary only at this point.)
Since $f_\infty(\Gamma)$ is a $\rho(\Gamma)$-orbit consisting 
of a single point, it is fixed by $\rho(\Gamma)$.  

For $\mu$ as above, denote by $\mu^n$ the $n$th convolution of $\mu$: 
\begin{equation*}
 \mu^n(\gamma, \gamma')= \sum_{\gamma_1 \in \Gamma} 
 \dots \sum_{\gamma_{n-1} \in \Gamma} \mu(\gamma, \gamma_1) \dots
 \mu(\gamma_{n-1},\gamma'). 
\end{equation*}
We define the {\it $n$-step energy} $E_{\mu^n,\rho}(f)$ of a $\rho$-equivariant
map $f$ by
\begin{equation*}
 E_{\mu^n,\rho}(f)= \frac{1}{2} \sum_{\gamma \in \Gamma}\mu^n(e,\gamma )
 d_Y(f(e),f(\gamma))^2.
\end{equation*}

\subsection{Examples of $n$-step energy}\label{example}

We first take a glance at examples of the computation of
$E_{\mu^n,\rho}(f)$. 
In what follows, we drop $\rho$ in $E_{\mu^n,\rho}$ and use the symbol 
$E_{\mu^n}$, unless no confusion is likely to occur. 

\begin{Example}
Let $\Gamma=\Z$ and $\mu$ the standard random walk on $\Z$:
\begin{equation*} 
\mu(k,l)= \begin{cases} 
\frac{1}{2} & \text{if }k-l =\pm1, \\ 0 & \text{otherwise.}
\end{cases}
\end{equation*}
Let $Y=\R$ and $\rho\colon \Z \longrightarrow \isom(\R)$ a homomorphism 
such that $\rho(1) (t) = u t + \tau$ for $t\in \R$, where $u = \pm 1$ and $\tau\in \R$. 
Let $f\colon \Z\longrightarrow \R$ be a $\rho$-equivariant map such that $f(0) = \alpha\in \R$. 
Then for $k\in \Z$, 
$$
\rho(k) (t) = \left\{ \begin{array}{cl}
t + k \tau & \mbox{if $u=1$}, \\
t & \mbox{if $u=-1$ and $k$ is even}, \\ 
-t+ \tau & \mbox{if $u=-1$ and $k$ is odd}, \end{array} \right. 
$$
and 
$$
f(k) = \left\{ \begin{array}{cl}
\alpha + k \tau & \mbox{if $u=1$}, \\
\alpha & \mbox{if $u=-1$ and $k$ is even}, \\ 
-\alpha+ \tau & \mbox{if $u=-1$ and $k$ is odd}. \end{array} \right. 
$$
Note that $\rho(\Z)$ has a global fixed point in $\R$ exactly when $u=1$ and $\tau=0$, 
or $u=-1$, and $f$ is harmonic exactly when $u = 1$, or $u = -1$ and $\alpha = \tau/2$. 

We now compute the $n$-step energy of $f$. 
Suppose that, among $n$ steps, a walker makes exactly $j$ steps to the right ($+1$). 
Then the walker should make $n-j$ steps to the left ($-1$), and he arrives at $2j-n \in \Z$. 
There are $\,_{n}C_{j}$ ways of such walks, each taking place with probability $(1/2)^n$. 
Therefore, 
\begin{eqnarray*} 
E_{\mu^n}(f) &=& \frac{1}{2} \sum_{k\in \Z} \mu^n(0, k)\, |f(k)-f(0)|^2 \\ 
&=& \frac{1}{2} \sum_{j=0}^n \frac{\,_{n}C_{j}}{2^n} |f(2j-n)-f(0)|^2 \\ 
&=& \left\{\begin{array}{cl} 
\frac{1}{2} \sum_{j=0}^n \frac{\,_{n}C_{j}}{2^n} (2j-n)^2 \tau^2 \,\, =\,\, \frac{n\tau^2}{2} & \mbox{if $u=1$}, \\ 
0 &\mbox{if $u=-1$ and $n$ is even}, \\ 
2 \left(\alpha - \frac{\tau}{2}\right)^2 &\mbox{if $u=-1$ and $n$ is odd}. 
\end{array}\right. 
\end{eqnarray*}
We conclude that $E_{\mu^n}(f) = n E_\mu(f)$ for all $n$ if $u=1$, 
and $E_{\mu^n}(f) \leq E_\mu(f)$ for all $n$ if $u=-1$. 
In the computation above for the $u=1$ case,  we have used the fact that
\begin{equation*}
\begin{split}
& \sum_{j=0}^n \frac{\,_{n}C_{j}}{2^n}j = (\text{the average of }B(n,1/2))= \frac{n}{2}, \\
& \sum_{j=0}^n \frac{\,_{n}C_{j}}{2^n}\left(j-\frac{n}{2}\right)^2 = (\text{the variance of }B(n,1/2))= \frac{n}{4}, 
\end{split}
\end{equation*}
where $B(n,1/2)$ denotes the symmetric binomial distribution. 
\end{Example}

As we will see in \S \ref{affine_case} (Corollary \ref{n-step_ineq2_affine}), 
when the target space $Y$ is a Hilbert space, 
$E_{\mu^n}(f) \leq nE_{\mu}(f)$ holds for any $\rho$-equivariant map $f$, 
and the equality holds if and only if $f$ is harmonic. 

\begin{Example} We take $\Gamma = F_m$ to be the free  group of rank $m$ generated by $s_1, \dots, s_m$. 
Let $S = \left\{ s_1^\pm,\dots,s_m^\pm \right\}$, and $\mu$ the standard random walk 
on $F_m$ with respect to the generator set $S$:
 \begin{equation*}
  \mu(\gamma,\gamma')=
  \left\{\begin{array}{cl} 
   \frac{1}{2m} & \text{if }\gamma' = \gamma s \text{ for some }s \in S, \\
   0 & \text{otherwise}.
\end{array}\right. 
 \end{equation*}
Clearly, $\mu$ is $F_m$-invariant, finitely supported, symmetric, and irreducible.  
 Let $Y$ be the Cayley graph of $F_m$ with respect to $S$.  Then $Y$ is
 a $2m$-regular tree.  We give a distance on $Y$ by setting the length
 of each edge to be $1$.  Let $\rho\colon F_m \longrightarrow \isom(Y)$ be the
 homomorphism that gives the action on $Y$ coming from the left action
 of $F_m$ on $F_m$ itself, and $f\colon F_m\longrightarrow Y$ the standard
 embedding of $F_m$ into its Cayley graph $Y$.  We give an estimate of
 $E_{\mu^n}(f)$. 
 We denote by $\mu^n(r)$ the probability of a walker on $F_m$ being at 
 distance $r$ from the starting point $e$ after taking $n$ steps following $\mu$. 
Thus 
 \begin{equation*}
   E_{\mu^n}(f) = \frac{1}{2} \sum_{r=0}^n \mu^n(r)r^2. 
 \end{equation*}
 Let $X_n$ be the Bernoulli walk on $\Z$ starting from
 $0$ which moves right with probability $p=(2m-1)/2m$ and left with
 probability $q=1/2m$.  
 Denote by $b^n(r)$ the probability that $X_n=r \in \Z$:
 \begin{equation*}
  b^n(r) =\,_n C_{(n+r)/2} \left(\frac{2m-1}{2m}\right)^{(n+r)/2}
  \left(\frac{1}{2m}\right)^{(n-r)/2}. 
 \end{equation*}
 Note that the average $\mathbb{E}(X_n)$ and the variance 
 $\mathbb{V}(X_n)$ are given by 
$$
\mathbb{E}(X_n) = n(p-q) = n(m-1)/m\quad \mbox{and}\quad \mathbb{V}(X_n) = 4npq = n(2m-1)/m^2, 
$$ 
respectively. 
Recall that $\mu^n(r) \leq b^n(r)/p = 2m\, b^n(r)/(2m-1)$ holds as explained in
 \cite{Silberman}.  Then, by the variance equality, 
\begin{eqnarray*} 
E_{\mu^n}(f) &\leq& \frac{1}{2} \sum_{r=0}^n \frac{2m}{2m-1} b^n(r)r^2\,\, =\,\, 
\frac{m}{2m-1}\mathbb{E}(X_n^2) \\ 
&=& \frac{m}{2m-1}\left(\mathbb{V}(X_n) + \mathbb{E}(X_n)^2\right)\,\, \leq \,\, 
\frac{m}{2m-1} n^2. 
\end{eqnarray*} 
\end{Example}

\subsection{General case}\label{general_case}

First we recall the well-known variance inequalities on a $\cat$ space. 

\begin{Lemma}\label{variance} 
Let $Y$ be a $\cat$ space with metric $d$. 
Let $\nu = \sum_{i=1}^m t_i \Dirac_{v_i}$ be a probability measure with finite support on $Y$ and $\overline{\nu}\in Y$ 
the barycenter of $\nu$. 
Then we have 
\begin{equation}\label{variance_iequality1}
\sum_{i=1}^m t_i d(v_i, w)^2 \geq \sum_{i=1}^m t_i d(v_i, \overline{\nu})^2 + d(\overline{\nu}, w)^2 
\end{equation}
for all $w\in Y$, and 
\begin{equation}\label{variance_iequality2} 
\frac{1}{2} \sum_{i=1}^m \sum_{j=1}^m t_i t_j d(v_i,v_j)^2 \geq \sum_{i=1}^m t_i d(v_i, \overline{\nu})^2. 
\end{equation}
\end{Lemma}

\begin{proof} 
We include a proof for the sake of completeness. 
Set $l = d(\overline{\nu}, w)$ and $F(w) = \sum_{i=1}^m t_i d(v_i, w)^2$. 
Let $c \colon [0,1] \longrightarrow Y$ be the (constant-speed) geodesic joining $\overline{\nu}$
and $w$; $c(0) = \overline{\nu}$, $c(1) = w$. 
By Proposition \ref{CAT(0)_ineq}, $d(v_i, c(\tau))^2 - (l\tau)^2$ is a convex function of $\tau$ , and hence 
the same is true of the function $\varphi(\tau) = F(c(\tau)) - (l\tau)^2$. 
Therefore, 
\begin{eqnarray*}
F(\overline{\nu}) - (l\tau)^2 &\leq& \varphi(\tau)\\
&\leq& (1-\tau) \varphi(0) + \tau \varphi(1)\\
&=& (1-\tau) F(\overline{\nu}) + \tau (F(w) - l^2),
\end{eqnarray*}
and so $\tau F(\overline{\nu}) + l^2 \tau(1-\tau) \leq \tau F(w)$.
Dividing the both sides by $\tau$ and letting $\tau\to 0$, we obtain \eqref{variance_iequality1}. 
\eqref{variance_iequality2} follows by integrating \eqref{variance_iequality1} against $d\nu(w)$. 
\end{proof}

We use this lemma to derive the following 

\begin{Lemma}\label{inner_product} 
Let $Y$ be a $\cat$ space and $p\in Y$. 
Let $\nu = \sum_{i=1}^m t_i \Dirac_{v_i}$ be a probability measure with finite support on $TC_pY$ 
and $\overline{\nu}\in TC_pY$ the barycenter of $\nu$. 
Then for any $w\in TC_pY$, we have 
\begin{equation}\label{inner_product_inequality}
\langle \overline{\nu}, w \rangle \geq \sum_{i=1}^m t_i \langle v_i, w \rangle. 
\end{equation}
The equality holds if $w= \overline{\nu}$. 
\end{Lemma}

\begin{proof} 
First we treat the $w= \overline{\nu}$ case  (by just revising the proof of Lemma 2.7 in \cite{Izeki-Nayatani}).  
Set $\psi(\tau) = \sum_{i=1}^m t_i d(v_i, \tau \overline{\nu})^2$, which takes its minimum at $\tau = 1$. 
On the other hand, we can rewrite 
$$
\psi(\tau) =  \sum_{i=1}^m t_i |v_i|^2 + \tau^2 |\overline{\nu}|^2 - 2\tau \sum_{i=1}^m t_i  \langle v_i, \overline{\nu} \rangle, 
$$
and the right-hand side takes its minimum at $\tau = \sum_{i=1}^m t_i  \langle v_i, \overline{\nu} \rangle/|\overline{\nu}|^2$ 
(if $\overline{\nu} \neq 0_p$, which we may assume). 
Therefore, 
\begin{equation}\label{inner_product_equality}
|\overline{\nu}|^2 = \sum_{i=1}^m t_i \langle v_i, \overline{\nu} \rangle. 
\end{equation}

Now for $w$ arbitrary, applying \eqref{variance_iequality1} to $TC_pY$ and rewriting in terms of the inner product 
on $TC_pY$ and using \eqref{inner_product_equality}, we obtain 
$$
\sum_{i=1}^m t_i |v_i|^2 + |w|^2 - 2 \sum_{i=1}^m t_i \langle v_i, w \rangle \geq 
\sum_{i=1}^m t_i |v_i|^2 + |w|^2 - 2 \langle \overline{\nu}, w \rangle. 
$$
Cancelling out the common expression on the both sides, we obtain \eqref{inner_product_inequality}. 
\end{proof}

We restate Lemma \ref{inner_product} in the form which we will use later. 

\begin{Lemma}\label{barycenter_lem}
Let $\Gamma$ be a finitely generated group equipped with a $\Gamma$-invariant, finitely supported 
and symmetric random walk $\mu$, and $Y$ a $\cat$ space. 
Suppose that a homomorphism $\rho\colon \Gamma \longrightarrow \isom(Y)$ is given, and 
let $f\colon \Gamma\longrightarrow Y$ be a $\rho$-equivariant map. 
Then for any $v \in TC_{f(e)}Y$, 
\begin{equation}\label{barycenter_ineq} 
 \langle -\Delta_\mu f(e), v \rangle \geq 
 \sum_{\gamma \in \Gamma} \mu(e,\gamma)\langle F_e(\gamma),  v \rangle 
\end{equation}
holds. 
\end{Lemma}

We now prove the following

\begin{Proposition}
 Let $\mu$, $\mu'$ be $\Gamma$-invariant, finitely supported 
 and symmetric random walks on $\Gamma$. Then, for any
 $\rho$-equivariant map $f$, 
 \begin{equation}\label{2-step_ineq}
  E_{\mu * \mu'}(f) \geq E_{\mu}(f) + E_{\mu'}(f)-
  \langle -\Delta_{\mu}f(e), -\Delta_{\mu'}f(e)\rangle
 \end{equation}
 holds, where $\mu* \mu'$ denotes the convolution of $\mu$ and $\mu'$. 
\end{Proposition}

\begin{proof}
Since $\pi_{f(\gamma)}\colon Y \longrightarrow TC_{f(\gamma)}Y$ is distance-nonincreasing, 
we obtain 
 \begin{eqnarray}\label{easy_computation}
   E_{\mu * \mu'}(f) 
   &=& \frac{1}{2}\sum_{\gamma, \gamma'} 
   \mu'(e,\gamma)\mu(\gamma,\gamma')d(f(e),f(\gamma'))^2 \nonumber\\
   &\geq& \frac{1}{2}\sum_{\gamma, \gamma'}
   \mu'(e,\gamma)\mu(\gamma,\gamma')
   d_{f(\gamma)}(F_{\gamma}(e),F_{\gamma}(\gamma'))^2 \nonumber\\
   &=& \frac{1}{2}\sum_{\gamma, \gamma'}
   \mu'(e,\gamma)\mu(\gamma,\gamma')
   \left(|F_{\gamma}(e)|^2+|F_{\gamma}(\gamma')|^2 
   -2\langle F_{\gamma}(e), F_{\gamma}(\gamma')\rangle \right) \\
   &=& \frac{1}{2} \sum_{\gamma}\mu'(e,\gamma) d(f(e),f(\gamma))^ 2 
   + \frac{1}{2} \sum_{\gamma}\mu'(e,\gamma)
   \sum_{\gamma'}\mu(e,\gamma^{-1}\gamma') \nonumber\\ 
   && \times\, d(f(e),f(\gamma^{-1}\gamma'))^2 
   - \sum_{\gamma, \gamma'} \mu'(e,\gamma)\mu(\gamma,\gamma')
     \langle F_{\gamma}(e), F_{\gamma}(\gamma') \rangle, \nonumber
 \end{eqnarray}
where we have used 
$|F_{\gamma}(\gamma')|=d(f(\gamma),f(\gamma')) = d(f(e),f(\gamma^{-1}\gamma'))$ 
and $\mu(\gamma,\gamma') = \mu(e,\gamma^{-1}\gamma')$; 
these follow from the $\rho$-equivariance of $f$ and the  $\Gamma$-invariance of $\mu$ respectively. 
The first and second terms in the last expression of \eqref{easy_computation} equal to $E_{\mu'}(f)$ 
and $E_{\mu}(f)$ respectively. 
On the other hand, using Lemma \ref{barycenter_lem} twice, we estimate the third term from below as 
  \begin{eqnarray}\label{use_lemma}
 \lefteqn{- \sum_{\gamma} \mu'(e,\gamma) \sum_{\gamma'}\mu(\gamma,\gamma') 
     \langle F_{\gamma}(e), F_{\gamma}(\gamma') \rangle} \nonumber \\
   &\geq & - \sum_{\gamma} \mu'(e,\gamma) 
        \langle F_{\gamma}(e), -\Delta_{\mu} f(\gamma)\rangle \\
   &= & - \sum_{\gamma} \mu'(e,\gamma^{-1}) 
        \langle F_{e}(\gamma^{-1}), -\Delta_{\mu} f(e)\rangle \nonumber \\
   &\geq & - \langle -\Delta_{\mu'}f(e), -\Delta_{\mu}f(e) \rangle. \nonumber 
  \end{eqnarray}
 To deduce the equality on the third line, one has to notice that
 $\rho(\gamma^{-1})$ induces an isometry $\rho(\gamma^{-1})_{*}\colon  
TC_{f(\gamma)}Y\longrightarrow TC_{f(e)}Y$, which maps $F_{\gamma}(e)$ and
 $-\Delta_{\mu} f(\gamma)$ to $F_{e}(\gamma^{-1})$ and 
 $-\Delta_{\mu} f(e)$ respectively. 
We have also used $\mu'(e,\gamma) = \mu'(e, \gamma^{-1})$. 
Combining these inequalities completes the proof. 
\end{proof}

\begin{Remark}\label{Hilbert_case}
Note that the difference between the both sides of \eqref{2-step_ineq} comes from 
the curvature of $Y$ and the nonlinearity of the tangent cones of $Y$. 
The former possibly makes the projection $Y\longrightarrow TC_pY$ distance-decreasing and causes
strict inequality in \eqref{easy_computation}. 
On the other hand, 
the latter may force the inequalities in \eqref{barycenter_ineq}, and thus in \eqref{use_lemma}, to become strict ones. 
In particular,  \eqref{2-step_ineq} becomes an equality when $Y$ is a Hilbert space. 
\end{Remark}

\begin{Corollary}\label{n-step_ineq}
 For any $\Gamma$-invariant, finitely supported and symmetric random walk
 $\mu$, $\rho$-equivariant map $f$, and positive integer $n$, 
 \begin{equation}\label{n-step_inequality} 
  E_{\mu^n}(f) \geq nE_{\mu}(f)- 
  \sum_{i=1}^{n-1}\langle -\Delta_i f(e), -\Delta_1 f(e) \rangle
 \end{equation}
 holds, where $-\Delta_i f(e)$ denotes 
the barycenter of $(F_e)_*(\mu^i(e,\cdot))$. 
\end{Corollary}

\begin{proof}
To prove by induction, suppose the inequality is true for $n-1$.  
Then by the proposition above
 \begin{eqnarray*}
E_{\mu^n}(f) &\geq & E_{\mu^{n-1}}(f)+E(f)- 
  \langle -\Delta_{n-1}f(e),-\Delta_1 f(e) \rangle \\
  &\geq & (n-1)E_{\mu}(f)- \sum_{i=1}^{n-2}
   \langle -\Delta_i f(e), -\Delta_1 f(e) \rangle 
   +E_{\mu}(f)\\ 
&& -\langle -\Delta_{n-1}f(e), -\Delta_1 f(e) \rangle \\
  &= & nE_{\mu}(f) - \sum_{i=1}^{n-1}
   \langle -\Delta_i f(e), -\Delta_1 f(e) \rangle. 
 \end{eqnarray*}
This completes the proof of Corollary \ref{n-step_ineq}.
\end{proof}

\begin{Remark}\label{remark_on_n-step_inequality} 
By the previus remark, \eqref{n-step_inequality} becomes an equality when $Y$ is a Hilbert space. 
See the next subsection for more on the Hilbertian case. 
\end{Remark}

\begin{Remark}\label{growth_for_general_case} 
If $f$ is harmonic, then we have $E_{\mu^{n}}(f) \geq n E_{\mu}(f)$, and the strict inequality 
possibly holds by the reason as explained in Remark \ref{Hilbert_case}. 
It is natural to expect that $E_{\mu^n}(f)/E_{\mu}(f)$ is bounded by a constant depending 
on some kind of growth rate of $Y$. 
As the following lemma shows, such a constant should not exceed $n^2$. 
\end{Remark} 

\begin{Lemma}\label{n-step_energy_and_gradient} 
 Let $\mu$ be a $\Gamma$-invariant, finitely supported and symmetric random
 walk on $\Gamma$, and $f\colon \Gamma\longrightarrow Y$ a $\rho$-equivariant map. 
Then the following estimates hold{\rm :} \\ 
{\rm (1)}\quad $|-\Delta_{\mu} f(e)|^2 \leq 2E_{\mu}(f)$. \\ 
{\rm (2)}\quad $E_{\mu^n}(f) \leq n^2 E_{\mu}(f)$. \\ 
{\rm (3)}\quad $|-\Delta_{\mu^n}f(e)|^2 \leq 2n^2E_{\mu}(f)$. 
\end{Lemma}

\begin{proof} 
We first prove (1). 
Using the variance inequality \eqref{variance_iequality1}, we obtain 
\begin{eqnarray*} 
|-\Delta_{\mu} f(e)|^2 &=& d_{TC_{f(e)}Y}(0_ {f(e)}, -\Delta_{\mu} f(e))^2 \\ 
&\leq& \sum_{\gamma\in \Gamma} \mu(e,\gamma)\, d_{TC_{f(e)}Y}(0_ {f(e)}, F_e(\gamma))^2 \\ 
&=& \sum_{\gamma\in \Gamma} \mu(e,\gamma)\, d_Y(f(e), f(\gamma))^2 \,\,=\,\, 2E_\mu(f). 
\end{eqnarray*} 

To prove (2), we compute
 \begin{eqnarray}\label{computation3}
 E_{\mu^n}(f) 
  &= & \frac{1}{2}\sum_{\gamma_1,\dots,\gamma_n}\mu(e,\gamma_1) \cdots
    \mu(\gamma_{n-1},\gamma_n)d(f(e),f(\gamma_n))^2 \nonumber \\
  &\leq & \frac{1}{2}\sum_{\gamma_1,\dots, \gamma_n}\mu(e,\gamma_1) \cdots
    \mu(\gamma_{n-1},\gamma_n) \left( \sum_{i=1}^n  d(f(\gamma_{i-1}),f(\gamma_i)) \right)^2 \\
  &\leq & \frac{1}{2}\sum_{\gamma_1,\dots \gamma_n}\mu(e,\gamma_1) \cdots
    \mu(\gamma_{n-1},\gamma_n) \cdot n \sum_{i=1}^n  d(f(\gamma_{i-1}),f(\gamma_i))^2, \nonumber
 \end{eqnarray}
where $\gamma_0 = e$. 
Note that by $\Gamma$-invariance of $\mu$ and the $\rho$-equivariance of $f$, 

 \begin{eqnarray*}
\lefteqn{\frac{1}{2}\sum_{\gamma_1, \dots, \gamma_n}\mu(e,\gamma_1) \dots 
     \mu(\gamma_{n-1}, \gamma_n)
     d(f(\gamma_{i-1}),f(\gamma_i))^2} \\
  & = & \frac{1}{2}\sum_{\gamma_{i-1}, \gamma_i, \gamma_n} 
     \mu^{i-1}(e, \gamma_{i-1}) \mu(\gamma_{i-1}, \gamma_i) \mu^{n-i}(\gamma_i, \gamma_n) 
     d(f(\gamma_{i-1}),f(\gamma_i))^2 \\
  & = & \frac{1}{2}\sum_{\gamma_{i-1}, \gamma_i} 
     \mu^{i-1}(e, \gamma_{i-1}) \mu(\gamma_{i-1}, \gamma_i) 
     d(f(\gamma_{i-1}),f(\gamma_i))^2 \\
  & = & \frac{1}{2}\sum_{\gamma_{i-1}, \gamma_i} 
     \mu^{i-1}(e, \gamma_{i-1}) \mu(e, \gamma_{i-1}^{-1}\gamma_i) 
     d(f(e),f(\gamma_{i-1}^{-1} \gamma_i))^2 \\
  & = & \sum_{\gamma_{i-1}} \mu^{i-1}(e, \gamma_{i-1}) E_{\mu}(f) 
   = E_{\mu}(f). 
 \end{eqnarray*}
 Together with \eqref{computation3}, this implies 
 $E_{\mu^n}(f) \leq n^2 E_{\mu}(f)$. 
 Now (3) follows from (1) and (2). 
\end{proof}

\begin{Proposition}
 Let $\mu$ be a $\Gamma$-invariant, finitely supported, symmetric
 random walk and $f\colon \Gamma \longrightarrow Y$ a $\rho$-equivariant map.  
Then
 \begin{equation}\label{n-step_ineq3}
  E_{\mu^n}(f) \geq nE_{\mu}(f) - \frac{n(n-1)}{2}\sqrt{2E_{\mu}(f)}
 \left|-\Delta_{\mu}(f) \right|
 \end{equation}
 holds. 
\end{Proposition}

\begin{proof}
 By Corollary \ref{n-step_ineq} and Lemma \ref{n-step_energy_and_gradient}
 (3), we obtain
 \begin{eqnarray*}
  E_{\mu^n}(f) & \geq & nE_{\mu}(f)- 
  \sum_{i=1}^{n-1}\langle -\Delta_i f(e), -\Delta_1 f(e) \rangle \\
  & \geq & nE_{\mu}(f)-\sum_{i=1}^{n-1} i\sqrt{2E_{\mu}(f)}|-\Delta_1 f(e)|.
 \end{eqnarray*}
 This implies \eqref{n-step_ineq3}. 
\end{proof}

We can now prove 

\begin{Theorem}[Gromov \cite{Gromov3}]\label{n-step}
Suppose there exist a positive integer $n$ and a positive real number $\varepsilon$
such that 
\begin{equation}\label{assumption_on_n-step_energy} 
E_{\mu^n}(f) \leq (n-\varepsilon)E_{\mu}(f)
\end{equation}
holds for any $\rho$-equivariant map $f\colon \Gamma\longrightarrow Y$. 
Then there exists a positive constant $C$ as in Proposition \ref{gradient}.
In particular, $\rho(\Gamma)$ admits a global fixed point. 
\end{Theorem}

\begin{proof}
Suppose \eqref{assumption_on_n-step_energy} holds for a $\rho$-equivariant map $f$. 
By \eqref{n-step_ineq3}, we see
\begin{equation*}
(n-\varepsilon)E_{\mu}(f)\geq  nE_{\mu}(f)- \frac{n(n-1)}{2}\sqrt{2E_{\mu}(f)} |-\Delta_{\mu}f(e)|, 
\end{equation*}
from which we get 
\begin{equation*}
 |-\Delta_{\mu}f(e)|^2 \geq
   \frac{2\varepsilon^2}{n^2(n-1)^2} E_{\mu}(f).
\end{equation*}
This completes the proof of Theorem \ref{n-step}. 
\end{proof}

In the next section, we will need the following result, which is slightly more general than 
the above theorem and follows immediately from its proof. 
\begin{Corollary}\label{first_proposition}\label{l-step} 
Suppose there exist a positive integer $n$ and a positive real number $\varepsilon$ 
satisfying the following condition{\rm :} 
for any $\rho$-equivariant map $f\colon \Gamma\longrightarrow Y$, there exists $l\leq n$ such that 
$$
E_{\mu^l}(f) \leq (l-\varepsilon) E_{\mu}(f). 
$$
Then there exists a positive constant $C$ as in Proposition \ref{gradient}.
In particular, $\rho(\Gamma)$ admits a global fixed point. 
\end{Corollary}

\subsection{Affine case}\label{affine_case}

Next we examine the behavior of $E_{\mu^n}(f)$ in the affine case, 
namely, the case when $Y$ is taken to be a real Hilbert space. 

Let $\Gamma$ be a finitely generated group and $\mu$ a $\Gamma$-invariant, finitely supported 
and symmetric random walk on $\Gamma$. 
Let $\rho\colon \Gamma \longrightarrow \isom(\mathcal{H})$ be a homomorphism 
and $f\colon \Gamma\longrightarrow \mathcal{H}$ a $\rho$-equivariant map, 
where $\mathcal{H}$ is a Hilbert space. 
Then $-\Delta_\mu f(\gamma)$ is given by 
\begin{equation*}
 -\Delta_\mu f(\gamma) = 
 \sum_{\gamma' \in \Gamma}\mu(\gamma, \gamma') (f(\gamma')-f(\gamma)).
\end{equation*}
Recall that according to the semi-direct product decomposition 
$\isom(\mathcal{H}) = \mathcal{O}(\mathcal{H}) \ltimes \mathcal{H}$, 
where $\mathcal{O}(\mathcal{H})$ is the orthogonal group of $\mathcal{H}$, 
$\rho$ is decomposed into 
the pair $(\rho_0, b)$ of a homomorphism $\rho_0\colon \Gamma \longrightarrow \mathcal{O}(\mathcal{H})$ 
and a map $b\colon \Gamma \longrightarrow \mathcal{H}$, 
so that $\rho(\gamma)v = \rho_0(\gamma)v+ b(\gamma)$ for $\gamma\in\Gamma$ and $v\in \mathcal{H}$.
We note that $-\Delta_\mu f \colon \Gamma \longrightarrow \mathcal{H}$ is $\rho_0$-equivariant; 
in fact, 
\begin{eqnarray*}
  -\Delta_\mu f(\gamma'\gamma)
 & =& \sum_{\gamma''\in \Gamma}
     \mu(\gamma'\gamma,\gamma'')(f(\gamma'')-f(\gamma'\gamma)) \\
 & =& \sum_{\gamma''\in \Gamma}
     \mu(\gamma,\gamma'^{-1}\gamma'')
     (\rho(\gamma')f(\gamma'^{-1}\gamma'')-\rho(\gamma')f(\gamma)) \\
 &=& \sum_{\gamma''\in \Gamma}
     \mu(\gamma,\gamma'^{-1}\gamma'')
     \rho_0(\gamma')(f(\gamma'^{-1}\gamma'')-f(\gamma)) \\
 &=& \rho_0(\gamma') (-\Delta_\mu f(\gamma)). 
\end{eqnarray*}
Note that we have used the $\Gamma$-invariance of $\mu$ and the linearity of $\rho_0(\gamma')$. 

As in the general case, the set of $\rho_0$-equivariant maps from $\Gamma$ to $\mathcal{H}$, 
denoted by $\mathcal{M}_{\rho_0}$, is identified with $\mathcal{H}$ through the correspondence 
$\mathcal{M}_{\rho_0}\ni \varphi \mapsto \varphi(e) \in \mathcal{H}$. 
An inner product on $\mathcal{M}_{\rho_0}$ is defined in a natural way
as 
\begin{equation*}
 \langle \varphi, \psi \rangle_{\mathcal{M}_{\rho_0}}
 := \langle \varphi(e), \psi(e) \rangle 
( = \langle \varphi(\gamma), \psi(\gamma) \rangle).
\end{equation*}
We define an averaging operator $M$ by
\begin{equation*}
 M \varphi(\gamma) = \sum_{\gamma' \in \Gamma}\mu(\gamma, \gamma')
  \varphi(\gamma'), \quad \varphi \in \mathcal{M}_{\rho_0}.
\end{equation*}
Since $\rho_0(\gamma)$ is linear, we see that 
$M\varphi \in \mathcal{M}_{\rho_0}$:
\begin{eqnarray*}
M \varphi(\gamma''\gamma)
& = &\sum_{\gamma' \in \Gamma}\mu(\gamma''\gamma, \gamma')
      \varphi(\gamma')
 = \sum_{\gamma' \in \Gamma}\mu(\gamma, \gamma''^{-1}\gamma')
     \rho_0(\gamma'')\varphi(\gamma''^{-1}\gamma')\\
& = & \rho_0(\gamma'')M\varphi(\gamma).
\end{eqnarray*}
Thus, $M$ is a linear operator acting on $\mathcal{M}_{\rho_0}\cong \mathcal{H}$. 
Since $\mu$ is symmetric and $\Gamma$-invariant, $M$ is selfadjoint: 
\begin{eqnarray*}
 \langle M\varphi, \psi \rangle_{\mathcal{M}_{\rho_0}}
&  =& \langle M\varphi(e), \psi(e) \rangle 
   = \sum_{\gamma}\mu(e,\gamma)
   \langle\rho_0(\gamma)\varphi(e), \psi(e)\rangle \\
&  =& \sum_{\gamma}\mu(e, \gamma^{-1})
   \langle \varphi(e), \rho_0(\gamma^{-1})\psi(e) \rangle \\
&  =& \langle \varphi, M\psi \rangle_{\mathcal{M}_{\rho_0}}. 
\end{eqnarray*}
Using this operator $M$, we can rewrite $-\Delta_n f$, where $-\Delta_n = -\Delta_{\mu^n}$, 
as follows.
\begin{eqnarray*}
  -\Delta_n f(\gamma) 
  &= & \sum_{\gamma' \in \Gamma} \mu^n(\gamma, \gamma')
      (f(\gamma')-f(\gamma)) \\
  &= & \sum_{\gamma_1, \gamma' \in \Gamma}
    \mu(\gamma,\gamma_1)\mu^{n-1}(\gamma_1,\gamma')
    (f(\gamma')-f(\gamma_1)+f(\gamma_1)-f(\gamma)) \\
  &= & \sum_{\gamma_1}\mu(\gamma,\gamma_1)
    \left(-\Delta_{n-1} f(\gamma_1) + f(\gamma_1)-f(\gamma)\right) \\
  &= & M(-\Delta_{n-1} f)(\gamma) + (-\Delta_1 f)(\gamma), 
\end{eqnarray*}
and thus, 
\begin{equation*}
 -\Delta_n f = M(-\Delta_{n-1}f)+(-\Delta_1 f).
\end{equation*}
Proceeding inductively, we see that
\begin{equation}\label{expression_for_Delta_n}
  -\Delta_n f = (M^{n-1}+M^{n-2} + \dots + M + I)(-\Delta_1 f).
\end{equation}
In particular, we see that if $f$ is $\mu$-harmonic, then $f$ must be
$\mu^n$-harmonic. 

\begin{Corollary}\label{n-step_ineq2_affine}
 Let $Y$ be a Hilbert space. 
 For any $\Gamma$-invariant, finitely supported and symmetric random walk $\mu$ on $\Gamma$, 
$\rho$-equivariant map $f\colon \Gamma \longrightarrow Y$, and a positive integer $n$, 
 \begin{equation*}
  E_{\mu^n}(f) \leq nE_{\mu}(f)
 \end{equation*}
 holds. The equality holds if and only if $f$ is harmonic.
\end{Corollary}

\begin{proof}
 According to Corollary \ref{n-step_ineq} and Remark \ref{remark_on_n-step_inequality}, 
it suffices to show $\langle -\Delta_i f(e), -\Delta_1 f(e) \rangle \geq 0$ for each $i$ in
 order to prove the inequality. 
 When $i=1$, this is obvious. 
 Suppose $i=2m+2$, $m\geq 0$. 
By \eqref{expression_for_Delta_n}, we get 
 \begin{eqnarray}\label{expansion}
  \langle -\Delta_i f(e), -\Delta_1 f(e) \rangle 
  &= &  \sum_{k=0}^m 
    \langle (M^{2k}+M^{2k+1})(-\Delta_1 f)(e), -\Delta_1 f(e)\rangle \\
  &= &  \sum_{k=0}^m 
    \langle (I+M)M^{k}(-\Delta_1 f)(e), M^{k}(-\Delta_1 f)(e)\rangle, \nonumber 
 \end{eqnarray}
 since $M$ is selfadjoint. 
Now for $\varphi \in \mathcal{M}_{\rho_0}$, 
 \begin{equation*}
  \langle (I+M)\varphi, \varphi \rangle_{\mathcal{M}_{\rho_0}} 
   =  \langle \varphi(e), \varphi(e)\rangle + \sum_{\gamma \in \Gamma}\mu(e,\gamma)
   \langle \rho_0(\gamma)\varphi(e), \varphi(e)\rangle  \geq 0,  
 \end{equation*}
where we have used $|\langle \rho_0(\gamma)\varphi(e), \varphi(e)\rangle| \leq |\varphi(e)|^2$
which holds since $\rho_0(\gamma)$ is orthogonal. 
Thus the operator $I+M$ is nonnegative,  
and applying this to \eqref{expansion}, we obtain 
 $\langle -\Delta_i f(e), -\Delta_1 f(e) \rangle \geq 0$. 

 Suppose $i=2m+3$, $m\geq 0$. Then (by \eqref{expression_for_Delta_n} again)
\begin{eqnarray*}
\lefteqn{\langle -\Delta_i f(e), -\Delta_1 f(e) \rangle} \\ 
&=&  \langle -\Delta_{i-1} f(e), -\Delta_1 f(e) \rangle 
+ \langle M^{m+1} (-\Delta_1 f(e)), M^{m+1}(-\Delta_1 f(e)) \rangle \\ 
&\geq&  0. 
\end{eqnarray*}

 Now suppose $E_{\mu^n}(f)=nE_{\mu}(f)$. Then 
 $\langle -\Delta_1 f(e), -\Delta_1 f(e) \rangle=0$, and hence $f$ is
 harmonic.  The converse is obvious. 
\end{proof}

\subsection{Converse of Theorem \ref{n-step}}\label{proposition_converse} 

The following proposition shows that an assertion slightly stronger than
the converse of Theorem \ref{n-step} holds for some $\cat$ spaces. 

\begin{Proposition}\label{converse_of_n-step}
 Let $Y$ be either a $\cat$ Riemannian manifold or an $\R$-tree, 
 and $\rho\colon \Gamma \longrightarrow \isom(Y)$  a homomorphism. 
 Suppose $\rho(\Gamma)$ admits a global fixed point.  Then there exists a
 positive constant $C_{\rho}$ such that 
 $E_{\mu^n}(f)\leq C_{\rho}E_{\mu}(f)$ for any $n \in \N$ and
 $\rho$-equivariant map 
 $f$. In particular, taking $n > C_{\rho}$, we obtain $n$, $\varepsilon$ 
 as in Theorem \ref{n-step}. 
\end{Proposition}

\begin{proof}
 Let $f\colon \Gamma \longrightarrow Y$ be a $\rho$-equivariant map. Denote by
 $F$ the fixed-point set of $\rho(\Gamma)$, and let
 $p_0 \in F$ be the nearest point from $f(e)$. Since $f(e) \in F$
 implies $E_{\mu}(f)=E_{\mu^n}(f)=0$, we may assume $f(e) \not\in F$. 
 Set $R=d(f(e),p_0)=d(f(e),F)$.  Since, for any
 $\gamma \in \Gamma$, $d(f(\gamma),p_0)=R$, 
 and hence $d(f(e), f(\gamma))\leq 2R$, 
 we see that
 $E_{\mu^n}(f)\leq 2R^2$ for any $n \in \N$. 
 Let $S=\{s \in \Gamma \mid \mu(e,s) \not= 0\}$. 
 Suppose that there exists a positive constant $\kappa$ such that 
 $\max\{\angle_{p_0}(f(e),f(s)) \mid s \in S\}\geq \kappa$ holds for
 any $\rho$-equivariant map $f$. 
 Then, for any $\rho$-equivariant map $f$, we have 
 $\max\{d(f(e),f(s))\mid s \in S\}\geq 2R \sin (\kappa/2)$. 
 This implies
 \begin{equation*}
  E_{\mu}(f) \geq 
  2R^2 \sin^2 \frac{\kappa}{2} \min_{s \in S}\mu(e,s), 
 \end{equation*}             
 and we can take
 \begin{equation*}
  C_{\rho}= 
  \left(\sin^2 \frac{\kappa}{2} \min_{s \in S}\mu(e,s)\right)^{-1}. 
 \end{equation*} 

 Suppose $Y$ is a Riemannian manifold. Note that $f(e)$ lies in a 
 geodesic starting from $p_0$ which is normal to $F$, and $p_0$ depends on $f$.
 So take any point $p \in F$, and set
\begin{equation*}
 \kappa_p = \inf_{V \in T_pF^{\perp}, |V|=1}
            \max \{\angle_p(V,\rho(s)_{*}V)\mid s \in S\},
\end{equation*}
 where $\rho(s)_{*}$ denotes the differential of $\rho(s)$, which
 induces an isometry on $T_pF^{\perp}$. 
 Note that $\kappa_p$ is positive. In fact, 
 since $T_pF^{\perp}$ is finite-dimensional, $\kappa_p=0$ implies the 
 existence of a unit vector $V \in T_pF^{\perp}$ fixed by $\rho(s)_{*}$ for any $s\in S$, 
and hence fixed by $\rho(\gamma)_{*}$ for any $\gamma\in \Gamma$ since $S$ generates $\Gamma$ 
by the $\Gamma$-invariance and the irreducibility of $\mu$. 
 Then the geodesic $\exp tV$ must be fixed by
 $\rho(\Gamma)$.  This contradicts the definition of $F$, since $V$ is
 normal to $F$. Let $q \in F$
 be another point in $F$. 
 Let $c\colon [0,1]\longrightarrow Y$ be the
 unique geodesic starting from $p$ and terminating at $q$, and 
 $P_t \colon T_pY \longrightarrow T_{c(t)}Y$ the
 parallel translation along $c$. 
 Note that $c$ must lie in $F$, and hence it is fixed by $\rho(\Gamma)$. 
 Therefore, for any $V\in T_pV$, $t \mapsto \rho(s)_{*}P_t(V)$ is a parallel vector field
 along $c$ with initial vector $\rho(s)_{*}P_0(V)=\rho(s)_{*}V$. By the
 uniqueness of a parallel vector field 
 with a given initial condition,
 we see
 that $\rho(s)_{*}P_1(V)=P_1(\rho(s)_{*}V)$, namely $\rho(s)_{*}$
 commutes with $P_1$. Thus the action of $\rho(s)$ on $T_qY$ is
 conjugate to that on $T_pY$ by $P_1$. In particular, 
 $\kappa_p = \kappa_q$, that is, $\kappa_p$ does not depend on the
 choice of $p \in F$.  Hence we can take $\kappa$ above to be
 $\kappa_p$.  

 Now suppose $Y$ is an $\R$-tree, and let $f$, $F$, $p_0$ and $R$ be as
 above. Since $Y$ is an $\R$-tree, the angle 
 $\angle_{p_0}(f(e),f(s))$ equals either $0$ or $\pi$.  
 Suppose there exists $s \in S$ such that 
 $[p_0,f(e)] \cap [p_0,f(s)]=\{p_0\}$, where $[p_0,q]$ denotes the
 geodesic segment joining $p_0$ and $q$.  
 Then $[p_0,f(e)] \cup [p_0,f(s)]$ is an arc (a topological segment)
 joining $f(e)$ and $f(s)$, 
 which must be unique in $Y$ by the definition of $\R$-tree. 
 In other words,
 $[p_0,f(e)] \cup [p_0,f(s)]$ is a geodesic segment joining $f(e)$ and
 $f(s)$. Therefore  $\angle_{p_0}(f(e),f(s))=\pi$.  Now assume the
 contrary: 
 $[p_0,f(e)] \cap [p_0,f(s)]\not= \{p_0\}$ for all $s \in S$. 
 Let $c_s\colon [0,R]\longrightarrow Y$, $s \in S$,  and 
 $c_e\colon [0,R]\longrightarrow Y$ be unit speed geodesics starting from $p_0$ 
 and terminating at $f(s)$ and $f(e)$ respectively.  By our assumption, 
 there exists a positive constant $T_s$ for each $s \in S$ such that 
 $c_s([0,T_s]) \subset c_e([0,R])$. 
 Since the geodesics are of unit 
 speed, this means $c_s|_{[0,T_s]}=c_e|_{[0,T_s]}$ for each $s \in S$. 
 Because $S$ is a finite set, we get a positive constant
 $T:= \min\{T_s \mid s \in S\}$.  By the definition of $T$, 
 $c_e|_{[0,T]}=c_s|_{[0,T]}$ for all $s \in S$ and 
 $c_e([0,T])\not= \{p_0\}$.
 It is clear that $c_e([0,T])$ must be fixed
 by $\rho(s)$ for all $s \in S$, and hence by $\rho(\Gamma)$.  
 This means that there is a fixed point $p =c(T)$ of 
 $\rho(\Gamma)$ which is closer to $f(e)$ than $p_0$. This contradicts
 our choice of $p_0$. 
 Therefore, for any $\rho$-equivariant map $f$, 
 $\max\{\angle_{p_0}(f(e),f(s)) \mid s \in S\}$ must be equal to
 $\pi$, and we can take $\kappa$ to be $\pi$. This completes the proof.  
\end{proof}

\begin{Remark}
 From the proof, one sees that $C_{\rho}$ for an $\R$-tree equals 
 $(\min_{s\in S} \mu(e,s))^{-1}$ and does not depend on $\rho$. 
 It is plausible that Proposition \ref{converse_of_n-step} is also true
 for Euclidean buildings. 
\end{Remark}

\section{Fixed-point property of random groups}

In this section, we will prove that a random group of Gromov's graph model 
associated with a sequence of expanders satisfying some additional conditions  
has fixed-point property for a certain large class of $\cat$ spaces. 

\subsection{Preliminaries on graphs}\label{finite_graph}
Let $G = (V, E)$ be a finite connected graph, where $V$ and $E$ are the sets of vertices 
and undirected edges, respectively. 
We denote the set of directed edges by $\overrightarrow{E}$. 
Let $\mu_G$ and $\nu_G$ denote the standard random walk on $G$ and the standard 
probability measure on $V$ given by 
$$
\mu_G(u,v)=\left\{\begin{array}{cl} \frac{1}{\deg(u)} & \mbox{if $\{u, v\} \in E$,}\\
0 & \mbox{otherwise,} \end{array} \right.\quad \mbox{and}\quad 
\nu_G(u)=\frac{\deg(u)}{2|E|}, 
$$
respectively. 
Note that $\mu_G$ is symmetric with respect to $\nu_G$: $\nu_G(u) \mu_G(u,v) = \nu_G(v) \mu_G(v,u)$. 
The {\em discrete Laplacian} $\Delta_G$ of $G$, acting on real-valued functions $\varphi$ on $V$, is defined by 
$$
(\Delta_G \varphi) (u) = \varphi(u) - \sum_{v\in V} \mu(u,v) \varphi(v),\quad u\in V.  
$$
Let $\lambda_1(G,\R)$ denote the second eigenvalue of $\Delta_G$. 
It is characterized variationally as 
$$
\lambda_1(G, \R) = \inf_\varphi \frac{\frac{1}{2} \sum_{u\in V} \nu_G(u) \sum_{v\in V}\mu_G(u,v) (\varphi(u)-\varphi(v))^2}{
\sum_{u\in V} \nu_G(u) (\varphi(u)- \overline{\varphi})^2},  
$$
where $\varphi$ is a nonconstant real-valued function on $V$, and $\overline{\varphi}$ denotes 
the average of $\varphi$, given by
$\overline{\varphi} = \left[ \sum_{u\in V} {\rm deg}(u) \varphi(u) \right] / \left[\sum_{u\in V} {\rm deg}(u) \right]$. 
The {\em girth} of $G$, denoted by $\girth(G)$, is the minimal length of a cycle (i.e.~a closed path) in $G$, 
and the {\em diameter} of $G$, denoted by $\diam(G)$, is the maximum distance between a pair of points in $G$. 

Let $\{ G_l = (V_l, E_l) \}_{l\in L}$ be a sequence of finite connected graphs with $L$ an unbounded set 
of positive integers and $|V_l|\to \infty$ as $l\to\infty$. 
We say that $\{ G_l  \}_{l\in L}$ is a {\em sequence of (bounded degree) expanders} if  it satisfies 
the following conditions for some positive integer $d_0$ and positive real number $\mu_0$: 
\begin{enumerate}
\renewcommand{\theenumi}{\roman{enumi}}
\renewcommand{\labelenumi}{(\theenumi)} 
\item $2\leq \deg(u)\leq d_0$ for all $l\in L$ and all $u\in V_l$, 
\item $\lambda_1(G_l, \R)\geq \mu_0$ for all $l\in L$. 
\end{enumerate}

\subsection{Graph-model random groups and their hyperbolicity}
We first recall the formulation of Gromov's graph model of random groups \cite{Gromov3}, \cite{Ollivier}. 
Let $\Gamma = F_k$ be the free group generated by $S = \{s_1^\pm,\dots,s_k^\pm\}$. 
Let $G = (V, E)$ be a finite connected graph, and we use the notations as in the previous subsection. 
A map $\alpha\colon \overrightarrow{E}\longrightarrow S$ satisfying $\alpha((u, v)) = \alpha((v, u))^{-1}$ 
for all $(u,v)\in \overrightarrow{E}$ is called an {\em $S$-labelling} of $G$. 
For such an $\alpha$ and a path $\overrightarrow{p} = (\overrightarrow{e}_1,\dots, \overrightarrow{e}_l)$ 
in $G$, where $\overrightarrow{e}_i \in \overrightarrow{E}$, define $\alpha(\overrightarrow{p}) 
= \alpha(\overrightarrow{e}_1)\cdot \dots \cdot \alpha(\overrightarrow{e}_l)\in \Gamma$. 
Then set $R_\alpha = \{ \alpha(\overrightarrow{c}) \mid \mbox{$\overrightarrow{c}$ is a cycle in $G$} \}$ 
and $\Gamma_{\alpha}=\Gamma/\overline{R_{\alpha}}$, where $\overline{R_{\alpha}}$ is the normal 
closure of $R_\alpha$. 
Let $\Lambda(G, k)$ denote the set of all $S$-labellings of $G$, consisting of $(2k)^{|E|}$ 
elements, and make it into a probability space by putting a uniform probability measure on it. 
When $|V|\to \infty$, the group $\Gamma_\alpha$ for a randomly and uniformly chosen $\alpha\in \Lambda(G, k)$ 
is a `random group'. 

To be precise, choose a sequence of finite connected graphs $\{ G_l = (V_l, E_l) \}_{l\in L}$ 
with $L$ an unbounded set of positive integers and $|V_l|\to \infty$ as $l\to\infty$. 
Given a group property P (e.g.~Kazhdan's property (T)), 
we say that a {\em random group has property P} if the probability of $\Gamma_\alpha$ having property P 
goes to one as $l\to \infty$, that is, if 
$|\{\alpha\in \Lambda(G_l, k) \mid \mbox{$\Gamma_\alpha$ has property P}\}| / |\Lambda(G_l, k)| \to 1$ 
as $l\to \infty$. 
In actual use, we primalily assume that $\{ G_l \}_{l\in L}$ is a sequence of expanders. 
In what follows, we make precise what kind of properties the expanders 
should have further, in order that the corresponding graph model is useful for our purpose. 

We begin with the specific example of expanders which was discovered by
Lubotzky, Phillips and Sarnak \cite{Lubotzky-Phillips-Sarnak}. 

\Example 
Let $p$ and $q$ be distinct primes which are congruent to $1$ modulo $4$. 
The {\em LPS expanders} $X^{p,q}$ are $(p+1)$-regular Cayley graphs of the group 
$\mathrm{PSL}(2, \F_q)$ if the Legendre symbol $\displaystyle \left(\frac{p}{q}\right) = 1$ 
and of $\mathrm{PGL}(2, \F_q)$ if $\displaystyle \left(\frac{p}{q}\right) = -1$, where
$\F_q$ is a finite field with $q$ elements ($\cong  \Z/q\Z$). 
They are so-called Ramanujan graphs, and also satisfy some other extremal 
combinatorial properties: \\
Case i. $\displaystyle \left(\frac{p}{q}\right) = -1$; $X^{p,q}$ is bipartite 
of order $n=|X^{p,q}| = q(q^2-1)$, 
\begin{enumerate} 
\renewcommand{\labelenumi}{(\alph{enumi})} 
\item $\girth(X^{p,q}) \geq 4\log_p q - \log_p 4$, \\
\item ${\rm diam}(X^{p,q}) \leq 2\log_p n + 2\log_p 2 + 1$.
\end{enumerate}
Case ii. $\displaystyle \left(\frac{p}{q}\right) = 1$; $n=|X^{p,q}| = q(q^2-1)/2$ 
and $X^{p,q}$ is not bipartite,  
\begin{enumerate} 
\renewcommand{\labelenumi}{(\alph{enumi})} 
\item $\girth(X^{p,q}) \geq 2\log_p q$, \\
\item ${\rm diam}(X^{p,q}) \leq 2\log_p n + 2\log_p 2 + 1$. 
\end{enumerate}
Let us introduce a new parameter $l= [\log_p q]$, where $p$ is fixed and $q$ varies, 
and set $G_l = X^{p,q}$. 
(Note that the map $q\mapsto l$ is not one-to-one. 
So for each $l$ we choose a single $q$ among those mapped to $l$.) 
Then in the both cases, the conditions (a), (b) are rewritten as 
$$
\girth(G_l) \geq {\rm const}_1 \cdot l\quad \mbox{and}\quad 
{\rm diam}(G_l) \leq {\rm const}_2 \cdot l 
$$
respectively. 
Note that one can choose ${\rm const}_1 = 2$ and ${\rm const}_2 = 6+ o_l(1)$.

\medskip
With this example as a model, we consider a sequence of finite connected graphs 
$\{ G_l = (V_l, E_l) \}_{l\in L}$ with $L$ an unbounded set of positive integers 
satisfying the following conditions for some positive integer $d_0$ and positive real number $\mu_0$: 
\begin{enumerate}
\renewcommand{\theenumi}{\roman{enumi}}
\renewcommand{\labelenumi}{(\theenumi)} 
\item $3\leq \deg(u)\leq d_0$ for all $l\in L$ and all $u\in V_l$, 
\item $\girth(G_l)\geq l$ and ${\rm diam}(G_l)\leq {\rm const}\cdot l$ for all $l\in L$, 
\item $\lambda_1(G_l, \R)\geq \mu_0$ for all $l\in L$. 
\end{enumerate}
For a fixed positive integer $j$, we also consider the graph $G_l^{(j)}$ obtained 
from $G_l$ by subdividing every edge of $G_l$ into $j$ edges by adding $j-1$ vertices. 
Set $l' = jl$ so that $l'$ varies over $jL$. 
Then the sequence of graphs $\{G_{l'/j}^{(j)}\}_{l'\in jL}$ satisfies the following conditions: 
\begin{enumerate}
\renewcommand{\theenumi}{\roman{enumi}}
\renewcommand{\labelenumi}{(\theenumi$\mbox{}'$)} 
\item $2\leq \deg(u)\leq d_0$ for all $l'\in jL$ and all $u\in V(G_{l'/j}^{(j)})$ , 
\item $\girth(G_{l'/j}^{(j)})\geq l'$ and ${\rm diam}(G_{l'/j}^{(j)})\leq {\rm const}\cdot l'$ 
for all $l'\in jL$, 
\item $\lambda_1(G_{l'/j}^{(j)}, \R)\geq c(\mu_0, j)>0$ for all $l'\in jL$.  
\end{enumerate}
(For (iii$\mbox{}'$), see \cite{Silberman}.) 
Moreover, if an arbitrary $\beta>1$ is given, then by choosing $j$ large enough, 
we can arrange so that $\{ G_{l'/j}^{(j)}  \}_{l'\in jL}$ satisfies \\
\quad (iv$\mbox{}'$)\,\, The number of embedded paths in $G_{l'/j}^{(j)}$ of length less than
$\frac{l'}{2}$ is less than \\
\phantom{\quad (iv)\,\,} ${\rm const}\cdot \beta^{l'/2}$. \\
(For this point, we refer the reader to \cite[p.~17]{Ghys}.) 

Henceforth, we will fix a sequence of finite connected graphs $\{ G_l \}_{l\in L}$ 
satisfying the conditions (i)-(iii). 
We will also fix a sufficiently large $j$, and consider the graph model of random groups 
associated with the sequence of graphs $\{G_{l'/j}^{(j)}\}_{l'\in jL}$. 
The fact that a random group of this model is an infinite group follows from 
the following theorem due to Gromov \cite{Gromov3} (see also \cite{Ghys}). 

\begin{Theorem}\label{hyperbolicity_theorem}
Let $\{ G_l = (V_l, E_l) \}_{l\in L}$ be a sequence of finite connected graphs with $L$ 
an unbounded set of positive integers. 
Suppose that $\{ G_l \}_{l\in L}$ satisfies the following conditions
for some positive integer $d_0$ and a choice of $\beta>1$ sufficiently close to $1${\rm :}  
\begin{enumerate} 
\renewcommand{\theenumi}{\roman{enumi}} 
\renewcommand{\labelenumi}{(\theenumi)} 
\item $2\leq \deg(u)\leq d_0$ for all $l\in L$ and all $u\in V_l$,  
\item $\girth(G_l)\geq l$ and ${\rm diam}(G_l)\leq {\rm const}\cdot l$ for all $l\in L$, 
\item the number of embedded paths in $G_l$ of length less than $\frac{l}{2}$ 
is less than ${\rm const}\cdot \beta^{l/2}$. 
\end{enumerate} 
Then a random group of the graph model associated with $\{ G_l \}_{l\in L}$ 
is non-elementary hyperbolic; in particular, it is an infinite group. 
\end{Theorem} 

\subsection{Fixed-point theorem} 

We first recall (see \cite{Wang}) 
\begin{Definition} 
For a finite connected graph $G$ and a $\cat$ space $T$, the {\em Wang invariant} 
$\lambda_1(G,T)$ is defined by $\lambda_1(G,T) = \inf \mathrm{RQ}(\varphi)$, 
where the infimum is taken over all nonconstant maps $\varphi\colon V\longrightarrow T$, and 
\begin{equation}\label{RQ} 
\mathrm{RQ}(\varphi) = \frac{\frac{1}{2}\sum_{u \in V}\nu_G(u)\sum_{v \in V}\mu_G(u,v)d_T(\varphi(u),\varphi(v))^2}{
\sum_{u \in V} \nu_G(u) d_T(\varphi(u), \mathrm{bar}(\varphi_*\nu_G))^2}. 
\end{equation} 
\end{Definition}

\begin{Theorem}\label{fixed_point_theorem} 
Given positive integers $k, d_0$ and positive real number $\lambda_0$, there exists $g_0 = g_0(\lambda_0)$ 
such that if $G=(V,E)$ is a finite connected graph and $\mathcal{Y}$ is a family of $\cat$ spaces satisfying  
\begin{enumerate}
\renewcommand{\theenumi}{\roman{enumi}}
\renewcommand{\labelenumi}{(\theenumi)}
\item $2\leq \deg(u)\leq d_0$ for all $u\in V$, 
\item $\girth(G)\geq g_0$, 
\item $\lambda_1(G, TC_pY) \geq \lambda_0$ for all $Y\in \mathcal{Y}$ and all $p\in Y$, 
\end{enumerate}
then with probability at least $1-a_1 e^{-a_2|V|}$, where $a_1 = a_1(k,\lambda_0)$ and 
$a_2 = a_2(k, d_0, \lambda_0)$, $\Gamma_\alpha$ has property F$\mathcal{Y}$.  
\end{Theorem}

The geometric part of the proof of the theorem is based on Corollary \ref{l-step}. 
We use it with the following setting: the group $\Gamma$ is the free group $F_k$ 
generated by $S = \{s_1^\pm,\dots,s_k^\pm\}$, and  the random walk $\mu$ is the standard one, 
that is, it is given by 
$$
\mu(\gamma, \gamma') = \left\{\begin{array}{cl} \frac{1}{2k} & \mbox{if $\gamma'=\gamma s$ for some $s\in S$,}\\ 
0 & \mbox{otherwise.} \end{array} \right.
$$

The probabilistic part of the proof of Theorem \ref{fixed_point_theorem} is based on 
the following proposition. 
A similar proposition was formulated and proved by Silberman \cite{Silberman} 
when the target space is a Hilbert space, in the course of detailing Gromov's argument 
in \cite[3.12]{Gromov3}. 
Our proof is simpler than Silberman's, and we will present it in the Appendix. 
(Our proof, however, is {\em less elementary} than Silberman's, as we replace 
his explicit calculation of binomial coefficients by use of the central limit theorem.)

\begin{Proposition}[cf.~{\cite[Proposition 2.14]{Silberman}}]\label{second_proposition} 
Suppose that $G = (V,E)$ is a finite connected graph and $n$ is a positive 
integer satisfying 
\begin{enumerate} 
\renewcommand{\theenumi}{\roman{enumi}} 
\renewcommand{\labelenumi}{(\theenumi)}
\item $2\leq \deg(u)\leq d$ for all $u\in V$, 
\item $2\leq n\leq \girth(G)/2$. 
\end{enumerate} 
Then with probability at least $1-a_1 e^{-a_2|V|}$, $a_1 = a_1(k,n)$, $a_2=a_2(k,d,n)$, 
the following assertion holds{\rm :} 
for any $\cat$ space $Y$, any homomorphism 
$\rho^{(\alpha)}\colon \Gamma_\alpha\longrightarrow {\rm Isom}(Y)$ 
and any $\rho^{(\alpha)}$-equivariant map $f^{(\alpha)}\colon \Gamma_\alpha\longrightarrow Y$, 
there exists an $l$ {\rm (}depending on $f^{(\alpha)}${\rm )}, $\sqrt{n} < l\leq n$, 
such that 
$$ 
E_{\mu^l, \rho}(f) \leq \frac{C}{\lambda_1(G, Y)} E_{\mu, \rho}(f),  
$$ 
where $\rho = \rho^{(\alpha)}\circ \mathrm{pr}$, $f = f^{(\alpha)}\circ \mathrm{pr}$ with 
$\mathrm{pr}$ denoting the projection from $\Gamma$ onto $\Gamma_\alpha$, 
and $C$ is an absolute constant. 
\end{Proposition} 

\noindent
{\em Proof of Theorem \ref{fixed_point_theorem}.}\quad 
For any $\cat$ space $Y$, it is easy to verify that $\lambda_1(G,Y) \geq \inf_{p\in Y} \lambda_1(G,TC_pY)$ 
(see \cite{Wang}). 
Therefore,  for any $Y\in \mathcal{Y}$, we have $\lambda_1(G,Y)\geq \lambda_0$. 

Now let $n$ be the minimum positive integer satisfying $C/\lambda_0 < \sqrt{n}$, and set $g_0 = 2n$. 
Then the assertion of Proposition \ref{second_proposition} holds with the high probability as stated there. 
Therefore, if $Y\in \mathcal{Y}$, we obtain 
$$
E_{\mu^l, \rho}(f) \leq \frac{C}{\lambda_0} E_{\mu, \rho}(f) \leq (l-\varepsilon) E_{\mu, \rho}(f), 
$$
where $\varepsilon = \sqrt{n} - C/\lambda_0$. 
By Corollary \ref{first_proposition}, $\rho(\Gamma) = \rho^{(\alpha)}(\Gamma_\alpha)$ 
fixes a point in $Y$. 
\qed

\medskip
Combining Theorem \ref{hyperbolicity_theorem} and Theorem \ref{fixed_point_theorem}, we obtain 

\begin{Theorem}\label{main_fixed_point_theorem1} 
Let $\{ G_l = (V_l, E_l) \}_{l\in L}$ be a sequence of finite connected graphs 
with $L$ an unbounded set of positive integers, and let $\mathcal{Y}$ be 
a family of $\cat$ spaces. 
Suppose that they satisfy the following conditions for some positive integer $d_0$, positive real number 
$\lambda_0$ and a choice of $\beta>1$ sufficiently close to $1${\rm :} 
\begin{enumerate}
\renewcommand{\theenumi}{\roman{enumi}}
\renewcommand{\labelenumi}{(\theenumi)}
\item $2\leq {\rm deg} (u)\leq d_0$ for all $l\in L$ and all $u\in V_l$, 
\item $\girth(G_l)\geq l$ and 
${\rm diam}(G_l)\leq {\rm const}\cdot l$ for all $l\in L$, 
\item $\lambda_1(G_l, TC_pY) \geq \lambda_0$ for all $l\in L$, all $Y\in \mathcal{Y}$ and all $p\in Y$,  
\item the number of embedded paths in $G_l$ of length less than $\frac{l}{2}$ 
is less than ${\rm const}\cdot \beta^{l/2}$. 
\end{enumerate} 
Then a random group of the graph model associated with $\{ G_l \}_{l\in L}$ is infinite hyperbolic 
and has property F$\mathcal{Y}$. 
\end{Theorem}

To formulate a class of $\cat$ spaces so that the condition (iii) of Theorem \ref{main_fixed_point_theorem1} 
is satisfied, we recall the definition of the invariant of a $\cat$ space introduced in \cite{Izeki-Nayatani}.

\begin{Definition}
Let $T$ be a $\cat$ space.
Let $\mu = \sum_{i=1}^m t_i\, {\rm Dirac}_{v_i}$ be a probability measure with finite support on $T$ 
and $\overline{\mu}\in T$ the barycenter of $\mu$.
Consider all maps $\iota\colon {\rm supp}\,\mu \longrightarrow \R^m$ satisfying
\begin{equation}\label{realization}
\Vert\iota(v_i)\Vert = d_T(\overline{\mu}, v_i),\quad \Vert\iota(v_i)-\iota(v_j)\Vert \leq d_T(v_i, v_j),
\end{equation}
and set
$$
\delta(\mu) = \inf_\iota  \left[ \biggl\Vert \int_T \iota(v)\, d\mu(v)\biggr\Vert^2 
\biggm/\int_T \Vert\iota(v)\Vert^2\, d\mu(v) \right] \in [0,1].
$$
We then define
$$
\delta(T) = \sup_\mu \delta(\mu)\in [0,1]. 
$$
Here, if we restrict the choices of $\mu$ to those with barycenter at a given $v \in T$, 
we denote the corresponding number by $\delta(T, v)$.
\end{Definition}

\begin{Theorem}\label{main_fixed_point_theorem1.5} 
Let $0\leq \delta_0< 1$, and let $\mathcal{Y}_{\leq\delta_0}$ denote the class of 
$\cat$ spaces $Y$ satisfying $\delta(TC_pY)\leq \delta_0$ for all $p\in Y$. 
Let $\{ G_l = (V_l, E_l) \}_{l\in L}$ be a sequence of finite connected graphs 
with $L$ an unbounded set of positive integers, satisfying the following conditions 
for some 
positive integer $d_0$, positive real number $\mu_0$ and a choice of $\beta>1$
sufficiently close to $1${\rm :} 
\begin{enumerate}
\renewcommand{\theenumi}{\roman{enumi}}
\renewcommand{\labelenumi}{(\theenumi)}
\item $3\leq {\rm deg} (u)\leq d_0$ for all $l\in L$ and all $u\in V_l$, 
\item $\girth(G_l)\geq l$ and 
${\rm diam}(G_l)\leq {\rm const}\cdot l$ for all $l\in L$, 
\item $\lambda_1(G_l, \R) \geq \mu_0$ for all $l\in L$, 
\item the number of embedded paths in $G_l$ of length less than $\frac{l}{2}$ 
is less than ${\rm const}\cdot \beta^{l/2}$. 
\end{enumerate} 
Then a random group of the graph model associated with $\{ G_l \}_{l\in L}$ is infinite hyperbolic 
and has property F$\mathcal{Y}_{\leq\delta_0}$. 
\end{Theorem}

\begin{proof} 
If $Y\in \mathcal{Y}_{\leq\delta_0}$, then by \cite[Proposition 5.3]{Izeki-Nayatani}, 
$$
\lambda_1(G_l,TC_pY)\geq (1-\delta(TC_pY)) \lambda_1(G_l, \R) \geq (1-\delta_0) \lambda_1(G_l, \R) 
$$
for all $p\in Y$. 
\end{proof} 

We now consider the sequence of graphs $\{G_{l'/j}^{(j)}\}_{l'\in jL}$ as in the previous subsection, 
where $j$ is chosen large enough so that the condition (iv$\mbox{}'$) is satisfied for a choice of $\beta>1$ 
sufficiently close to $1$. 
The graph $G_{l'/j}^{(j)}$ satisfies the condition (iii$\mbox{}'$) : $\lambda_1(G_{l'/j}^{(j)}, \R)\geq c(\mu_0, j)>0$. 

\begin{Corollary}\label{main_fixed_point_theorem2} 
Let $0\leq \delta_0< 1$, and let $\mathcal{Y}_{\leq\delta_0}$ denote the class of 
$\cat$ spaces $Y$ satisfying $\delta(TC_pY)\leq \delta_0$ for all $p\in Y$. 
Let $\{ G_l = (V_l, E_l) \}_{l\in L}$ be a sequence of finite connected graphs 
with $L$ an unbounded set of positive integers, satisfying the following conditions for some 
positive integer $d_0$ and positive real number $\mu_0${\rm :} 
\begin{enumerate}
\renewcommand{\theenumi}{\roman{enumi}}
\renewcommand{\labelenumi}{(\theenumi)}
\item $3\leq {\rm deg} (u)\leq d_0$ for all $l\in L$ and all $u\in V_l$, 
\item $\girth(G_l)\geq l$ and 
${\rm diam}(G_l)\leq {\rm const}\cdot l$ for all $l\in L$, 
\item $\lambda_1(G_l, \R) \geq \mu_0$ for all $l\in L$. 
\end{enumerate} 
For each $l\in L$, let $G_l^{(j)}$ be the $j$-subdivision of $G_l$, and set $l' = jl$. 
Here, $j$ is chosen large enough so that $\{G_{l'/j}^{(j)}\}_{l'\in jL}$ satisfies\\ 
\quad {\rm (iv$\mbox{}'$)}\,\, the number of embedded paths in $G_{l'/j}^{(j)}$ of length less than
$\frac{l'}{2}$ is less than \\
\phantom{\quad (iv$\mbox{}'$)\,\,} ${\rm const}\cdot \beta^{l'/2}$ \\
for a choice of $\beta>1$ suffciently close to $1$. 
Then a random group of the graph model associated with $\{G_{l'/j}^{(j)}\}_{l'\in jL}$ is infinite hyperbolic 
and has property F$\mathcal{Y}_{\leq\delta_0}$. 
\end{Corollary}

\begin{proof} 
One has only to verify that $\lambda_1(G_{l'/j}^{(j)},TC_pY)$ is bounded from below by a positive constant, 
independent of $l'$, $Y$ and $p$. 
As was already noted, we have $\lambda_1(G_{l'/j}^{(j)}, \R) \geq c(\mu_0, j)$ for all $l'\in jL$. 
Therefore, as in the predeeding proof, 
$$
\lambda_1(G_{l'/j}^{(j)},TC_pY) \geq (1-\delta_0) c(\mu_0, j), 
$$
getting the desired estimate. 
\end{proof}

\begin{Remark} 
With the notations and assumptions as in Theorem \ref{main_fixed_point_theorem1}, it is 
plausible that $\lambda_1(G_l^{(j)}, TC_pY)\geq c(\lambda_0, j)>0$ holds for all 
$l\in L$, all $y\in \mathcal{Y}$ and all $p\in Y$. If this was the case, we would 
obtain a version of Theorem \ref{main_fixed_point_theorem1} for the sequence of graphs 
$\{G_{l'/j}^{(j)}\}_{l'\in jL}$. 
\end{Remark}

\section{Distortion and the invariant $\delta$}\label{section_delta}

In view of the assumption for $\cat$ spaces in Corollary \ref{main_fixed_point_theorem2}, 
it is important to estimate the invariant $\delta$ of the tangent cones of a $\cat$ space. 
In this section, we first give an upper bound of $\delta(T, 0_T)$, where $T$ is a $\cat$ 
metric cone with cone point $0_T$, in terms of the radial distortion (defined below) of $T$. 
We then estimate the radial distortion of the tangent cones of some Euclidean buildings. 

We begin with some definitions. 

\begin{Definition} 
(1)\,\, For a metric space $S$, let $T = C(S) = (S \times \R_{\geq 0}) / (S \times \{0\})$.
Define a distance $d_T$ on $T$ by $d_{T}(v, v') = t^2 + {t'}^2 - 2tt'\cos \min \{ d_S(u,u'), \pi \}$, 
where $v=(u,t), v'=(u',t') \in T$. 
The metric space $(T, d_T)$  is called the {\it metric cone} over $S$. \\
(2)\,\, Let $D_{\rm rad}(T)$ denote the infimum number $D$ satisfying the following condition: 
there exists a map $\iota\colon T\longrightarrow \mathcal{H}$, where $\mathcal{H}$ is a Hilbert 
space, such that 
\begin{equation}\label{radial}
\iota(v) = t\cdot \iota(u)\quad \mbox{with}\quad \Vert\iota(u)\Vert=1
\end{equation} 
and 
\begin{equation}\label{bilipschitz} 
\frac{1}{D}\cdot d_T(v,v') \leq \Vert\iota(v) - \iota(v')\Vert \leq d_T(v,v') 
\end{equation} 
for all $v=(u,t), v'=(u',t') \in T$. 
If no such map exists, then we define $D_{\rm rad}(T) = \infty$. 
The number $D_{\rm rad}(T)$ is called the {\it radial distortion} of $T$.
Note that $D_{\rm rad}(T)$ is not less than the usual distortion 
(cf.~\cite{Matousek}) of $T$. 
\end{Definition}

\begin{Lemma}\label{estimate_by_distortion}
Let $T$ be a $\cat$ space, and $\mu$ a finite-support probability measure on $T$. 
Let $\iota\colon T\longrightarrow \mathcal{H}$ be a $1$-Lipschitz map, where $\mathcal{H}$ 
is a Hilbert space. 
Then we have 
$$
\int_T \Vert\iota(v)-\overline{\iota_*\mu}\Vert^2\, d\mu(v) 
\geq \frac{1}{D(\iota)^2} \int_T d_T(v, \overline{\mu})^2\, d\mu(v), 
$$
where $D(\iota)$ is the distortion of $\iota$, that is, the minimum number $D$ such that 
\eqref{bilipschitz} holds for all $v, v'\in T$. 
\end{Lemma}

\begin{proof}
By using \eqref{variance_iequality2} and the fact that the inequality becomes an equality 
for a Hilbert space, we obtain 
\begin{eqnarray*}
\int_T \Vert\iota(v)-\overline{\iota_*\mu}\Vert^2\, d\mu(v) 
&=& \frac{1}{2} \int_\mathcal{H} \Vert\iota(v)-\iota(w)\Vert^2\, d\mu(v) d\mu(w) \\
&\geq& \frac{1}{D(\iota)^2}\cdot \frac{1}{2} \int_T d_T(v, w)^2\, d\mu(v) d\mu(w) \\
&\geq& \frac{1}{D(\iota)^2} \int_T d_T(v, \overline{\mu})^2\, d\mu(v). 
\end{eqnarray*}
\end{proof}

\begin{Proposition}\label{estimate_of_delta_by_distortion}
Let $T$ be a $\cat$ metric cone with cone point $0_T$. 
Then we have 
$$
\delta(T, 0_T) \leq 1 - \frac{1}{{D_{\rm rad}(T)}^2}. 
$$
\end{Proposition}

\begin{proof}
Let $\iota \colon T \longrightarrow \mathcal{H}$ be a map with the properties 
\eqref{radial} and \eqref{bilipschitz}. 
Let $\mu$ be a finite-support probability measure on $T$ such that $\overline{\mu} = 0_T$. 
Then 
$$
\delta(\mu) \leq \frac{\Vert\int_T \iota(v)\, d\mu(v)\Vert^2}{\int_T \Vert\iota(v)\Vert^2\, d\mu(v)} 
= \frac{\Vert\overline{\iota_*\mu}\Vert^2}{\int_T \Vert\iota(v)\Vert^2\, d\mu(v)}. 
$$
On the other hand, 
$$
\int_T \Vert\iota(v)-\overline{\iota_*\mu}\Vert^2\, d\mu(v) = \int_T \Vert\iota(v)\Vert^2\, d\mu(v) 
- \Vert\overline{\iota_*\mu}\Vert^2. 
$$
Therefore, 
$$
\delta(\mu)\leq 1 - \frac{\int_T \Vert\iota(v)-\overline{\iota_*\mu}\Vert^2\, d\mu(v)}{
\int_T \Vert\iota(v)\Vert^2\, d\mu(v)}\leq 1 - \frac{1}{D(\iota)^2}
$$
by Lemma \ref{estimate_by_distortion}, and the proposition follows. 
\end{proof}

\begin{Remark} 
It should be useful in future study to have an estimate of the Wang invariant $\lambda_1(G,T)$ from below. 
Indeed, we can show that 
\begin{equation}\label{lambda_wang_distortion} 
\lambda_1(G,T) \geq \frac{1}{D(T)^2} \lambda_1(G, \R), 
\end{equation} 
where $D(T)$ denotes the (usual) distortion of $T$, 
holds for a finite connected graph $G$ and any $\cat$ space $T$ which is not 
necessarily a cone. 
By the variance inequality \eqref{variance_iequality2}, the denominator of \eqref{RQ} is estimated as 
$$
\sum_{u \in V} \nu_G(u) d_T(\varphi(u), \mathrm{bar}(\varphi_*\nu_G))^2 
\leq \frac{1}{2} \sum_{u,v \in V}\nu_G(u)\nu_G(v) d_T(\varphi(u),\varphi(v))^2, 
$$
and therefore, 
\begin{equation}\label{intermediate}
\mathrm{RQ}(\varphi) \geq 
\frac{\frac{1}{2}\sum_{u \in V}\nu_G(u)\sum_{v \in V}\mu_G(u,v)d_T(\varphi(u),\varphi(v))^2}{
\frac{1}{2} \sum_{u,v \in V}\nu_G(u)\nu_G(v) d_T(\varphi(u),\varphi(v))^2}.
\end{equation} 
Now suppose  $\iota\colon T\longrightarrow \mathcal{H}$ is a map satisfying \eqref{bilipschitz}, 
where $\mathcal{H}$ is a Hilbert space. 
Then clearly, 
\begin{eqnarray*}
\lefteqn{\mbox{the right-hand side of \eqref{intermediate}}}\\ 
&\geq& \frac{1}{D^2}
\frac{\frac{1}{2}\sum_{u \in V}\nu_G(u)\sum_{v \in V}\mu_G(u,v)\Vert(\iota\circ\varphi)(u)-(\iota\circ\varphi)(v)\Vert^2}{
\frac{1}{2} \sum_{u,v \in V}\nu_G(u)\nu_G(v) \Vert(\iota\circ\varphi)(u)-(\iota\circ\varphi)(v)\Vert^2} \\ 
&=&  \frac{1}{D^2}
\frac{\frac{1}{2}\sum_{u \in V}\nu_G(u)\sum_{v \in V}\mu_G(u,v)\Vert(\iota\circ\varphi)(u)-(\iota\circ\varphi)(v)\Vert^2}{
\sum_{u \in V} \nu_G(u) \Vert(\iota\circ\varphi)(u)- \mathrm{bar}((\iota\circ\varphi)_*\nu_G)\Vert^2} \\
&\geq& \frac{1}{D^2} \lambda_1(G, \mathcal{H}) = \frac{1}{D^2} \lambda_1(G, \R),  
\end{eqnarray*} 
and we conclude \eqref{lambda_wang_distortion}. 
Note that we have used the fact that the variance inequality \eqref{variance_iequality2} becomes 
an equality for a Hilbert space. 
\end{Remark} 

As mentioned in the Introduction, fixed-point property for Euclidean buildings 
(of certain types) are of particular interest. 
In the remainder of this section, we will estimate the radial distortion 
of the tangent cones of some Euclidean buildings. 

A building is a simplicial complex which is the union of a family of subcomplexes, 
called apartments, satisfying a certain set of axioms (see \cite{Brown}). 
One of the axioms requires that the apartments are isomorphic to a Coxeter complex 
of the same type. 
Here a Coxeter complex is a certain simplicial complex canonically associated 
with a Coxeter group; e.g. the Coxeter complex of the symmetric group $S_{n+1}$ 
is isomorphic to a triangulated $(n-1)$-sphere. 
A building is called Euclidean if its apartments are isomorphic to 
a Euclidean Coxeter complex; e.g. the Coxeter complex of type $\widetilde{A}_n$, 
which is associated with the group $S_{n+1} \ltimes \left( \mathbb{Z}^{n+1}/\mathbb{Z} (1,\ldots, 1) \right)$ 
and is isomorphic to a triangulated Euclidean $n$-space. 
A Euclidean building can be equipped with a distance by transplanting the Euclidean 
distance onto each apartment, and the building becomes a $\cat$ space 
with this distance (see \cite[Chapter 6]{Brown}).

Henceforth, we restrict our attention to the Euclidean building associated with 
the simple algebraic group $\mathrm{PGL}(n+1, \mathbb{Q}_r)$, where $r$ is a prime and 
$\mathbb{Q}_r$ is the $r$-adic number field. 
Let $Y_{n, r}$ denote this building; it is $n$-dimensional, and its apartments are 
simplicially isometric to the Euclidean Coxeter complex of type $\widetilde{A}_n$. 

If $n=1$, $Y_{1,r}$ is a regular tree of degree $r+1$ with all edges having equal length. 
If $p$ is an interior point of an edge, then the tangent cone at $p$ is isometric 
to a line, whose radial distortion and $\delta$ take trivial values. 
Suppose that $p$ is a vertex. 
Then the tangent cone at $p$ is isometric to the ($r+1$)-pod $P_{r+1}$, which is the union 
of $r+1$ half-lines with all endpoints identified. 
The radial distortion of $P_{r+1}$ is realized by arranging it in $\R^r$ so that the half-lines 
pass through the vertices of a regular $r$-simplex, and thus $D_{\rm rad}(P_{r+1}) = \sqrt{2r/(r+1)}$. 
On the other hand, $\delta(P_{r+1}) = 0$ as verified in \cite[p.~172, Example 3]{Izeki-Nayatani}. 

If $n=2$, $Y_{2,r}$ is two-dimensional, and its apartments are simplicially isometric 
to the Euclidean plane with equilateral triangulation. 
Simplicially, the links of its vertices are all isomorphic to the same generalized triangle of degree $r+1$, 
which is a regular bipartite graph of degree $r+1$ with $2(r^2+r+1)$ vertices and will be denoted 
by $\mathcal{G}_r$. 
Metrically, this means that the tangent cone at $p\in Y_{2,r}$ is isometric to 
a Euclidean plane if $p$ is an interior point of a maximal simplex, 
to the product of ($r+1$)-pod $P_{r+1}$ with a line if $p$ is an interior point of an edge, 
and to the metric cone $C(\mathcal{G}_r)$ over the graph $\mathcal{G}_r$ equipped with a distance 
by assigning length $\pi/3$ to each edge, if $p$ is a vertex. 
In the first case the values of the invariants in question are trivial, 
while in the second case they are identical to those of $P_{r+1}$. 
Therefore, it remains to examine the third case that $p$ is a vertex of $Y_{2,r}$, 
in which case, it is known \cite{Izeki-Nayatani} that 
$$
\delta(C(\mathcal{G}_r)) \geq 
\frac{(\sqrt{r}-1)^2}{2(r-\sqrt{r}+1)}. 
$$ 
In fact, let $\mu_0$ be the probability measure on $C(\mathcal{G}_r)$ given by 
$\mu_0 = \sum_{i=1}^N \frac{1}{N}\, {\rm Dirac}_{e_i}$, 
where $N=2(r^2+r+1)$ and $e_i$, $i = 1,\dots,N$, are the vertices of $\mathcal{G}_r$. 
Then the barycenter of $\mu_0$ coincides with the cone point of $C(\mathcal{G}_r)$, 
and we showed that 
$$
\delta(\mu_0) = 
\frac{(\sqrt{r}-1)^2}{2(r-\sqrt{r}+1)}. 
$$
In order to verify the `$\leq$'-part of this equality, we \cite[\S 7]{Izeki-Nayatani} introduced a certain 
family of $1$-Lipschitz embeddings of the cone $C(\mathcal{G}_r)$ into Euclidean spaces. 
We now recall these embeddings, and then use them to estimate the radial distortion of $C(\mathcal{G}_r)$. 

Let $V = \oplus_{i=1}^N \R e_i$ be the real vector space having the vertices $e_i$ as formal basis vectors. 
Note that there is a natural inclusion $C(\mathcal{G}_r) \hookrightarrow V$. 
We consider all positive semidefinite inner products $\langle\cdot,\cdot\rangle$ on $V$ 
whose value $\langle e_i, e_j \rangle$ on each pair of vertices $e_i, e_j$ depends only on the combinatorial 
distance $d_{\mathcal{G}_r}(e_i, e_j)$ between these vertices. 
We also require that the inner products of adjacent vertices are the same as those in $C(\mathcal{G}_r)$. 
Thus we consider symmetric bilinear forms $\langle\cdot,\cdot\rangle_{a,b}$ on $V$ defined by 
$$
\langle e_i, e_j \rangle_{a,b} = \left\{\begin{array}{ccc} 
1 &\mbox{if}& d_{\mathcal{G}_r}(e_i,e_j)=0,\\ 
1/2 &\mbox{if}& d_{\mathcal{G}_r}(e_i,e_j)=1,\\ 
a &\mbox{if}& d_{\mathcal{G}_r}(e_i,e_j)=2,\\ 
b &\mbox{if}& d_{\mathcal{G}_r}(e_i,e_j)=3, 
\end{array} \right. 
$$
and restrict the parameters $a, b$ to the range where $\langle\cdot,\cdot\rangle_{a,b}$ is 
positive semidefinite. 
For $a,b$ in this range, consider the sum of the eigenspaces belonging to the positive eigenvalues 
of the Gram matrix $G_{a,b} = (\langle e_i, e_j \rangle_{a,b})$, and let $W_{a,b}$ be the corresponding 
subspace of $V$. 
Restricted on $W_{a,b}$ the inner product $\langle\cdot,\cdot\rangle_{a,b}$ is positive definite, 
and the natural projection $V\longrightarrow W_{a,b}$ preserves the inner products. 
The composition of the maps $C(\mathcal{G}_r) \hookrightarrow V \rightarrow W_{a,b}$ gives a map 
from $C(\mathcal{G}_r)$ into the Euclidean space $W_{a,b}$. 
We denote this map by $\iota_{a,b}$; it is radial and $1$-Lipschitz; it is also isometric 
when restricted to the cone over each edge of $\mathcal{G}_r$.

It is easy to see that the distortion of $\iota_{a,b}$ is computed as 
\begin{eqnarray}\label{map_distortion}
D(\iota_{a,b}) &=& \max \left\{ \sqrt{\frac{2-2\cos (2\pi/3)}{2-2a}}, \sqrt{\frac{2-2\cos \pi}{2-2b}} \right\} \nonumber \\ 
&=& \max \left\{ \sqrt{\frac{3}{2-2a}}, \sqrt{\frac{2}{1-b}} \right\}. 
\end{eqnarray} 
On the other hand, $G_{a,b}$ can be readily related to the adjacency matrix of $\mathcal{G}_r$, 
whose eigenvalues were computed by Feit and Higman \cite{Feit-Higman}. 
It follows that the eigenvalues of $G_{a,b}$ are given by 
$$
(r^2+r+1)(a\pm b) + (1-a) \pm (1/2-b)(r+1)\quad \mbox{with multiplicities $1$} 
$$
and
$$
(1-a)\pm (1/2-b)\sqrt{r}\quad \mbox{with multiplicities $r^2+r$}.
$$
Under the constraint that these are nonnegative, the quantity \eqref{map_distortion} takes 
its minimum with 
$$
a = \frac{r-1-\sqrt{r}}{2r},\quad b = \frac{r^2-r-(r+1)\sqrt{r}}{2r^2}.
$$
(Incidentally, these values coincide with those giving the optimal upper bound of $\delta(\mu_0)$.) 
The minimum value $2r/\sqrt{(r+1)(r+\sqrt{r})}$ gives an upper bound of $D_{\rm rad}(C(\mathcal{G}_r))$. 
Note, in particular, that $D_{\rm rad}(C(\mathcal{G}_r)) < 2$ for all primes $r$. 
With the above values of $a,b$, $G_{a,b}$ has positive eigenvalues $r^2 + 1 - (r+1)\sqrt{r}$, $(r+1+\sqrt{r})/r$ 
with multiplicities $1$, $r^2+r$ respectively and zero eigenvalue with multiplicity $r^2+r+1$. 
Therefore, $W_{a,b}$ has dimension $r^2+r+1$. 
Observe that as $r$ tends to infinity, the above values of $a, b$ both approach $1/2$. 
This means that when $r$ is large, the images of the vertices of $\mathcal{G}_r$ in $W_{a,b}$ 
are nearly at equidistance to one another. 

We now treat the case of general $n$. 
Analogously to the $n=2$ case, if $p\in Y_{n,r}$ is not a vertex, then 
the tangent cone of $Y_{n,r}$ at $p$ 
is isometric to a metric cone of the form $\prod_{i=1}^m T_{k_i,r}\times \R^l$, where $T_{k_i,r}$ 
is the tangent cone of $Y_{k_i,r}$ at a vertex, $l>0$ and $\sum_{i=1}^m k_i + l = n$. 
Since $D_\mathrm{rad}(\prod_{i=1}^m T_{k_i,r}\times \R^l) = \max_{1\leq i\leq m} D_\mathrm{rad}(T_{k_i,r})$, 
we assume henceforth that $p$ is a vertex of $Y_{n,r}$. 
Then the tangent cone at $p$ is isometric to the metric cone $C(\mathcal{S}_{n,r})$ 
over the spherical building $\mathcal{S}_{n,r}$ associated with the finite group $\mathrm{PGL}(n+1, \F_r)$, 
and the apartments of $\mathcal{S}_{n,r}$ are simplicially isometric to the tessellated unit $(n-1)$-sphere 
associated with the symmetric group $S_{n+1}$. 
A chamber of $\mathcal{S}_{n,r}$ has $n$ vertices $e_1,\dots, e_n$, and the distances between them 
measured by the metric of $C(\mathcal{S}_{n,r})$ are given by 
$$
d_{C(\mathcal{S}_{n,r})}(e_i, e_j) = \sqrt{2 - 2 \sqrt{[i(n+1-j)]/[j(n+1-i)]}}, 
$$ 
when the vertices are  appropriately ordered. 
The minimum of these distances is 
$$
d_{\min} = \left\{\begin{array}{cl} 
\sqrt{2 - 2 \sqrt{(n-1)/(n+3)}} & \mbox{if $n$ is odd},\\
\sqrt{2 - 2n/(n+2)} = 2/\sqrt{n+2} & \mbox{if $n$ is even}.
\end{array}\right. 
$$ 
Motivated by the observation we made at the end of the preceding paragraph, we construct an embedding 
of $C(\mathcal{S}_{n,r})$ into a Euclidean space as follows. 
Denote the number of vertices of $\mathcal{S}_{n,r}$ by $N$. 
First take a regular simplex $\sigma$ with $N$ vertices in $\R^N$ whose vertices are located in the unit sphere 
with center at the origin and at the distance $d_{\min}$ to one another. 
Next map one by one the vertices of $\mathcal{S}_{n,r}$ to those of $\sigma$, and then extend it naturally 
to a radial embedding of $C(\mathcal{S}_{n,r})$. 
Clearly, this embedding, which we call $\iota$, is $1$-Lipschitz and have the same distortion 
as that of its restriction to $\mathcal{S}_{n,r}$. 
To estimate the distortion, we have to bound the ratio $d_{C(\mathcal{S}_{n,r})}(v, v')/\Vert\iota(v)-\iota(v')\Vert$ 
from above over all pairs of distinct points $v, v'$ in $\mathcal{S}_{n,r}$. 
It is easy to see that if we vary $v, v'$, the above ratio is maximized when they are at vertices 
of $\mathcal{S}_{n,r}$. 
Since $\mathcal{S}_{n,r}$ is a building, we may also assume that $v, v'$ are in the same apartment 
of $\mathcal{S}_{n,r}$. 
Therefore, the problem is reduced to bounding $d_{C(\mathcal{S}_{n,r})}(v, v')/\Vert\iota(v)-\iota(v')\Vert$ 
from above over all pairs of distinct vertices $v, v'$ in a fixed apartment of $\mathcal{S}_{n,r}$. 
Now this ratio is clearly bounded from above by $2/d_{\min}$, which therefore gives an upper bound 
of the distortion of $\iota$. 
Note that the constant $2/d_{\min}$ is monotone increasing with $n$, and diverges to infinity as $n\to \infty$. 

We record the consequence of the preceding discussion as 

\begin{Proposition}\label{tangent_cone_distortion}
The radial distortion and the invariant $\delta$ 
of all tangent cones of the Euclidean building $Y_{n,r}$ are bounded from above by 
$$
\left\{\begin{array}{cl} 
2/\sqrt{2 - 2 \sqrt{(n-1)/(n+3)}} & \mbox{if $n$ is odd},\\
\sqrt{n+2} & \mbox{if $n$ is even} 
\end{array}\right. 
$$
and 
$$
\left\{\begin{array}{cl} 
\left(2 +2 \sqrt{(n-1)/(n+3)}\right)/4 & \mbox{if $n$ is odd},\\
(n+1)/(n+2) & \mbox{if $n$ is even} 
\end{array}\right. 
$$
respectively. 
\end{Proposition} 
 
\begin{proof} 
Let $T$ be a tangent cone of $Y_{n,r}$.  
Then $D_{\rm rad}(T)$ is bounded as stated, and so is $\delta(T, 0_T)$. 
For $v\neq 0_T$, the tangent cone $TC_vT$ is isometric to the product 
of lower dimensional cones and possibly a Euclidean space. 
Since $\delta(T, v)\leq \delta(TC_vT, 0_v)$ (see \cite[Lemma 6.2]{Izeki-Nayatani}) 
and the invariant $\delta$ behaves in the same way as the radial distortion 
for the product, we conclude that $\delta(T, v)$, and hence $\delta(T)$, 
is also bounded as stated. 
\end{proof} 

\begin{Remark}
As noted above, the upper bound in the proposition diverges to infinity as $n\to \infty$. 
The embedding $\iota$ does not preserve the shape of chambers of $\mathcal{S}_{n,r}$, 
and one might expect that by constructing an embedding so that it preserves the shape of chambers 
of $\mathcal{S}_{n,r}$, one would get a better upper bound. 
However, numerical test done for $n=3$ indicates that the upper bound so obtained should 
diverge to infinity as $r\to \infty$ even though $n$ is kept bounded. 
\end{Remark}

Combining the proposition above with Corollary \ref{main_fixed_point_theorem2}, we obtain 

\begin{Theorem}\label{main_fixed_point_theorem3} 
For a fixed positive integer $N$, 
let $\mathcal{B}_{\leq N}$ denote the family of all the Euclidean buildings $Y_{n,r}$ with $n\leq N$ 
and $r$ arbitrary prime. 
Let $\{ G_l = (V_l, E_l) \}_{l\in L}$ be a sequence of finite connected graphs 
with $L$ an unbounded set of positive integers, satisfying the following conditions for some 
positive integer $d_0$ and positive real number $\mu_0${\rm :} 
\begin{enumerate}
\renewcommand{\theenumi}{\roman{enumi}}
\renewcommand{\labelenumi}{(\theenumi)}
\item $3\leq {\rm deg} (u)\leq d_0$ for all $l\in L$ and all $u\in V_l$, 
\item $\girth(G_l)\geq l$ and 
${\rm diam}(G_l)\leq {\rm const}\cdot l$ for all $l\in L$, 
\item $\lambda_1(G_l, \R) \geq \mu_0$ for all $l\in L$. 
\end{enumerate} 
For each $l\in L$, let $G_l^{(j)}$ be the $j$-subdivision of $G_l$, and set $l' = jl$. 
Here, $j$ is chosen large enough so that $\{G_{l'/j}^{(j)}\}_{l'\in jL}$ satisfies\\ 
\quad {\rm (iv$\mbox{}'$)}\,\, the number of embedded paths in $G_{l'/j}^{(j)}$ of length less than
$\frac{l'}{2}$ is less than \\
\phantom{\quad (iv)\,\,} ${\rm const}\cdot \beta^{l'/2}$ \\
for a choice of $\beta>1$ suffciently close to $1$. 
Then a random group of the graph model associated with $\{G_{l'/j}^{(j)}\}_{l'\in jL}$ is 
infinite hyperbolic and has property F$\mathcal{B}_{\leq N}$. 
\end{Theorem}

\section{Appendix}

In this Appendix, we will prove Proposition \ref{second_proposition}. 
Let $G = (V,E)$ be a finite connected graph, and $Y$ a $\cat$ space. 
For a map $\varphi\colon V \longrightarrow Y$ and a positive integer $n$, the {\em $n$-step energy} 
of $\varphi$ is defined by 
\begin{equation*}
E_{\mu_G^n}(\varphi) = \frac{1}{2} \sum_{u \in V}\nu_G(u) \sum_{v \in V} \mu_G^n(u,v) 
d_Y(\varphi(u),\varphi(v))^2,  
\end{equation*}
where $\mu_G$ is the standard random walk on $G$ and $\nu_G$ is the standard probability measure 
on $V$ (cf.~\S \ref{finite_graph}). 
We have the following 

\begin{Lemma}\label{spectral_gap}
For any map $\varphi\colon V\longrightarrow Y$ and any positive integer $n$, 
we have  
\begin{equation*}
E_{\mu_G^n}(\varphi) \leq \frac{2}{\lambda_1(G,Y)} E_{\mu_G}(\varphi). 
\end{equation*}
\end{Lemma}

\begin{proof}
 Let $\overline{\varphi} = \bary(\varphi_*\nu_G)$. 
 Using the triangle inequality and the symmetry of $\mu_G^n$ with
 respect to $\nu_G$, we obtain
\begin{eqnarray}\label{for_n-energy}
 E_{\mu_G^n}(\varphi) &\leq & \frac{1}{2}
   \sum_{u \in V}\nu_G(u)\sum_{v \in V}\mu_G^n(u,v)
         \left(d_Y(\varphi(u),\overline{\varphi}) + d_Y(\varphi(v),\overline{\varphi})\right)^2 \nonumber\\
  &\leq & \frac{1}{2}
      \sum_{u \in V}\nu_G(u)\sum_{v \in V} 
      \mu_G^n(u,v)\left(
      2d_Y(\varphi(u),\overline{\varphi})^2
    + 2d_Y(\varphi(v),\overline{\varphi})^2 \right)\\
  &= & 2 \sum_{u \in V}\nu_G(u)\sum_{v \in V} 
      \mu_G^n(u,v)d_Y(\varphi(u),\overline{\varphi})^2 \nonumber\\
  &= & 2\sum_{u \in V} \nu_G(u) d_Y(\varphi(u),\overline{\varphi})^2. \nonumber
\end{eqnarray}
On the other hand, by the definition of $\lambda_1(G,Y)$, we have
\begin{equation}\label{from_def_of_lambda_1}
\sum_{u \in V}\nu_G(u)d_Y(\varphi(u),\overline{\varphi})^2 
 \leq  \frac{1}{\lambda_1(G,Y)} E_{\mu_G}(\varphi).
\end{equation}
Combining \eqref{for_n-energy} and \eqref{from_def_of_lambda_1}, we
 obtain the desired inequality. 
\end{proof}

Let $\Gamma = F_k$ be the free group generated by $S = \{s_1^\pm,\dots,s_k^\pm\}$, 
and let $\Gamma$ act on itself from the left. 
Let $\alpha\colon \overrightarrow{E}\longrightarrow S$ be an $S$-labelling of $G$. 
Recall that associated with $\alpha$ is the group $\Gamma_{\alpha}=\Gamma/\overline{R_{\alpha}}$, 
where $R_\alpha = \{ \alpha(\overrightarrow{c}) 
\mid \mbox{$\overrightarrow{c}$ is a cycle in $G$} \}$ and $\overline{R_{\alpha}}$ is its normal closure. 
As in \cite{Silberman}, we will exclusively work on $\Gamma$ rather than on $\Gamma_\alpha$. 

For each positive integer $n$, define the `push-forward' of $\mu_G^n$ with respect to $\alpha$ by
\begin{eqnarray*}
\overline{\mu}_{\Gamma,\alpha}^n (\gamma, \gamma') 
&=& \sum_{u\in V} \nu_G(u) \sum_{|\overrightarrow{p}| = n, p_0 = u, 
\gamma \alpha(\overrightarrow{p}) = \gamma'} \mu_G^n(\overrightarrow{p}) \\
&=& \sum_{|\overrightarrow{p}| = n, 
\gamma \alpha(\overrightarrow{p}) = \gamma'} \nu_G(p_0) \mu_G^n(\overrightarrow{p}), 
\end{eqnarray*}
where $p_0$ is the initial vertex of $\overrightarrow{p}$. 
Note that $\overline{\mu}_{\Gamma,\alpha}^n$ is a $\Gamma$-invariant random walk on $\Gamma$. 
For a homomorphism $\rho^{(\alpha)}\colon \Gamma_\alpha\longrightarrow {\rm Isom}(Y)$ 
and a $\rho^{(\alpha)}$-equivariant map $f^{(\alpha)}\colon \Gamma_\alpha\longrightarrow Y$, 
set $\rho = \rho^{(\alpha)}\circ {\rm pr}$ and $f = f^{(\alpha)}\circ {\rm pr}$, 
where ${\rm pr}$ is the natural projection from $\Gamma$ onto $\Gamma_\alpha$. 
Then $f$ is a $\rho$-equivariant map, for which we can transplant the estimate of Lemma \ref{spectral_gap} 
to obtain 
\begin{equation}\label{transplanted_spectral_gap}
E_{\overline{\mu}_{\Gamma,\alpha}^n}(f) \leq \frac{2}{\lambda_1(G,Y)} E_{\overline{\mu}_{\Gamma,\alpha}}(f) 
\end{equation}
for all positive integers $n$ (cf.~{\cite[p.~155 -- p.~156]{Silberman}}). 

For $n, \gamma, \gamma'$ fixed, regard $\overline{\mu}_{\Gamma,\alpha}^n (\gamma, \gamma')$ as a random variable 
of $\alpha$, and denote its expectation by $\overline{\mu}_{\Gamma,G}^n (\gamma, \gamma')$.
We have the following lemma, which compares $\overline{\mu}_{\Gamma,G}^n$ 
with the standard random walk $\mu_\Gamma$ on $\Gamma$, given by 
$$
\mu_\Gamma(\gamma, \gamma') = \left\{\begin{array}{cl} \frac{1}{2k} & \mbox{if $\gamma' = \gamma s$ for some $s\in S$,}\\ 
0 & \mbox{otherwise.} \end{array} \right.
$$

\begin{Lemma}[cf.~{\cite[Lemma 2.12]{Silberman}}]\label{weighted_sum}
Suppose that $\deg(u)\geq 2$ for all $u\in V$, and choose a positive 
integer $n$ so that $n < \mathrm{girth}(G)/2$. 
Then there exist weights $P_G^n(l)\geq 0$ with $\sum_{l=0}^nP_G^n(l)=1$, independent of $\gamma$, $\gamma'$, 
such that 
$$
\overline{\mu}_{\Gamma,G}^n (\gamma, \gamma') = \sum_{l=0}^n P_G^n(l)\mu_\Gamma^l(\gamma, \gamma'). 
$$
Moreover, there exists an absolute constant $C<1$ such that 
$$
 Q^n_G:=\sum_{l \leq \sqrt{n}} P^n_G(l) \leq C 
$$
unless $n=1$. 
\end{Lemma}

\begin{proof}
Note that any ball of radius $n$ in $G$ is a tree; this is the most fundamental fact 
for the whole proof. 
The former part of the lemma can be proved by following Silberman's argument almost verbatim, 
and the weights $P_G^n(l)$ are given by 
$$
P_G^n(l) = \sum_{u\in V} \nu_G(u) P_{G,u}^n(l), 
$$
where $P_{G,u}^n(l)$ is the probability that an $n$-step random walk starting from $u$ 
reaches a vertex at distance $l$ from $u$. 
Here we prove the latter part of the lemma by an argument simpler than that proposed by Silberman. 
To do this, consider the standard Bernoulli walk on $\Z$, and let $b^n(r)$ denote the probability that an 
$n$-step walk starting from zero reaches an integer less than or equal to $r$ in absolute value. 
Since $\deg(u)\geq 2$ for all $u\in V$, the random walk on $G$ travels further than the Bernoulli 
walk on $\Z$. 
More precisely, we have 
\begin{equation*}
\sum_{l \leq \sqrt{n}}P^n_{G,u}(l) \leq b^n(\sqrt{n}). 
\end{equation*}
We now recall that the $n$-step Bernoulli walk has variance $n$. 
Then by the central limit theorem, we obtain 
\begin{equation*}
b^n(\sqrt{n}) \underset{n\to \infty}\longrightarrow \int_{-1}^1 \frac{1}{\sqrt{2\pi}}e^{-x^2/2} dx < 1. 
\end{equation*}
Therefore, there exists $C<1$ such that 
\begin{equation*}
 \sum_{l \leq \sqrt{n}}P^n_{G,u}(l) \leq C 
\end{equation*}
for all $n$ (other than $1$). 
Averaging over $u$, we conclude the latter assertion of the lemma. 
\end{proof}

The following lemma also can be proved by going on the same lines as Silberman's proof of 
\cite[Lemma 2.13]{Silberman}, which applies a general result on the concentration of measure 
to the random variable $\overline{\mu}_{\Gamma,\alpha}^n (\gamma , \gamma')$ defined 
on the set all $S$-labellings $\alpha$. 

\begin{Lemma}[cf.~{\cite[Lemma 2.13]{Silberman}}]\label{probability} 
In addition to the assumptions of Lemma \ref{weighted_sum}, suppose that 
$\deg(u)\leq d$ for all $u\in V$. 
Then with probability at least $1 - a_1 e^{-a_2 |V|}$, where $a_1 = a_1(k,n)$, $a_2 = a_2(k,d,n)$, 
we have 
$$
\overline{\mu}_{\Gamma,\alpha}^n (\gamma, \gamma') \geq \frac{1}{2} \overline{\mu}_{\Gamma,G}^n (\gamma, \gamma')
\quad \mbox{and}\quad 
\overline{\mu}_{\Gamma,\alpha} (\gamma, \gamma') \leq \mu_\Gamma (\gamma, \gamma')
$$
for all $\gamma,\gamma'\in \Gamma$. 
\end{Lemma}

With the ingredients above, the proof of Proposition \ref{second_proposition} proceeds 
as in \cite[Proof of Proposition 2.14]{Silberman}. 
We include it for the sake of completeness. 

\medskip\noindent
{\em Proof of Proposition \ref{second_proposition}.}\quad 
If follows from Lemma \ref{probability} that 
$$
E_{\overline{\mu}_{\Gamma,\alpha}^n} (f) \geq \frac{1}{2} E_{\overline{\mu}_{\Gamma,G}^n} (f),\quad
E_{\overline{\mu}_{\Gamma,\alpha}} (f) \leq E_{\mu_\Gamma} (f)
$$
hold with the probability as in the statement of the proposition. 
By combining this with \eqref{transplanted_spectral_gap}, we obtain 
$$
E_{\overline{\mu}_{\Gamma,G}^n} (f) \leq \frac{4}{\lambda_1(G,Y)} E_{\mu_\Gamma} (f) 
$$
with the same probability. 
By Lemma \ref{weighted_sum}, we can estimate $E_{\overline{\mu}_{\Gamma,G}^n} (f)$ from below: 
\begin{eqnarray*}
E_{\overline{\mu}_{\Gamma,G}^n} (f) &=& \sum_{l=0}^n P_G^n(l) E_{\mu_\Gamma^l} (f) 
\geq \sum_{\sqrt{n} < l\leq n} P_G^n(l) E_{\mu_\Gamma^l} (f) \\
&\geq& \left( \sum_{\sqrt{n} < l\leq n} P_G^n(l) \right) E_{\mu_\Gamma^{l_0}} (f), 
\end{eqnarray*}
where $E_{\mu_\Gamma^{l_0}} (f) = \min \{ E_{\mu_\Gamma^l} (f) \mid \sqrt{n} < l\leq n \}$. 
We also have $\sum_{\sqrt{n} < l\leq n} P_G^n(l) = 1-Q_G^n \geq 1-C$. 
Therefore, we conclude that 
\begin{eqnarray*}
E_{\mu_\Gamma^{l_0}} (f) \leq \frac{1}{1-C}\, \frac{4}{\lambda_1(G,Y)} E_{\mu_\Gamma} (f) 
\end{eqnarray*} 
holds with high probability. 
\qed

\bigskip\noindent
{\bf Added in proof.}\quad 
During the submission of the present paper, we learned that Naor and Silberman \cite{Naor-Silberman}  
proved that the graph-model random group had fixed-point property for a family of $p$-uniformly convex 
geodesic metric spaces with a certain Poincar\'{e}-type constant uniformly bounded. 
For a family of $\cat$ spaces (which are $2$-uniformly convex), this condition is equivalent to the 
uniformly-boundedness of 
the Wang invariant (the condition (iii) in Theorem \ref{main_fixed_point_theorem1}). 
However, our Theorem \ref{main_fixed_point_theorem3}, the fixed-point theorem for a family of Euclidean buildings 
with dimensions bounded from above, does not follow from their result. 

Let $\mathcal{Y}_{< 1}$ denote the class of $\cat$ spaces $Y$ satisfying 
$\sup_{p\in Y} \delta(TC_pY) < 1$, which contains all of the Euclidean buildings 
$Y_{n,r}$ (cf.~\S 4 for the notation). 
By using our Corollary \ref{main_fixed_point_theorem2}, it is shown that the group of Theorem 7.7 of 
\cite{Arzhantseva-Delzant}, called the Gromov monster, has fixed-point property 
for $\mathcal{Y}_{< 1}$. 
It should be mentioned that the Gromov monster has fixed-point property for a 
larger class of metric spaces, as shown by combining Theorem 1.2 of \cite{Naor-Silberman}
with  Theorem 7.7 of \cite{Arzhantseva-Delzant}. 
It is also worthwhile to mention that Kondo \cite{Kondo2} has found examples of $\cat$ space $Y$ for which 
$\sup_{p\in Y} \delta(TC_pY) = 1$.

\end{document}